\documentclass[12pt]{amsart}
\usepackage{amsfonts}
\usepackage{amsmath, amsthm, amssymb,color}
\usepackage{latexsym}
\usepackage{txfonts}
\usepackage{bm}

\numberwithin{equation}{section}
\allowdisplaybreaks

\newtheorem{lemma}{Lemma}[section]

\newtheorem{theorem}[lemma]{Theorem}

\newtheorem{proposition}[lemma]{Proposition}
\newtheorem{definition}[lemma]{Definition}
\newtheorem{corollary}[lemma]{Corollary}
\newtheorem{example}[lemma]{Example}
\newtheorem{exercise}[lemma]{Exercise}
\newtheorem{remark}[lemma]{Remark}

\usepackage{graphics}
\usepackage{graphicx}
\usepackage{epsfig}

\newcommand{\bth}{\begin{theorem}}
\newcommand{\ethe}{\end{theorem}}

\newcommand{\bre}{\begin{remark}\em }
\newcommand{\ere}{\end{remark}}

\newcommand{\ble}{\begin{lemma}}
\newcommand{\ele}{\end{lemma}}

\newcommand{\pp}{point process}
\newcommand{\bde}{\begin{definition}}
\newcommand{\ede}{\end{definition}}
\newcommand{\bco}{\begin{corollary}}
\newcommand{\eco}{\end{corollary}}

\newcommand{\bpr}{\begin{proposition}}
\newcommand{\epr}{\end{proposition}}

\newcommand{\bexer}{\begin{exercise}}
\newcommand{\eexer}{\end{exercise}}

\newcommand{\bexam}{\begin{example}\rm }
\newcommand{\eexam}{\end{example}}

\newcommand{\bali}{\begin{align}}
\newcommand{\eali}{\end{align}}

\renewcommand\d{{\mathrm d}}

\newcommand\wt{\widetilde}
\newcommand\N{\mathbb N}

\newcommand\R{\mathbb R}

\begin{document}

\bibliographystyle{alpha}
\title[Prediction of components in random sums]
      {Prediction of components in random sums}
\today
\author[M. Matsui]{Muneya Matsui}
\address{Department of Business Administration, Nanzan University,
18 Yamazato-cho Showa-ku Nagoya, 466-8673, Japan}
\email{mmuneya@nanzan-u.ac.jp}
%\author[T. Rolski]{Tomasz Rolski}
%\address{Mathematical Institute,
%The University of Wroclaw,
%pl. Grunwaldzki 2/4,
%50-384 Wroclaw, Poland}
%\email{rolski@math.uni.wroc.pl}

\begin{abstract} 
 %Motivated by prediction problems in a Poisson 
 %shot noise process, 
 We consider predictions of the random
 number and the magnitude of each iid component in a random sum based on
 its distributional structure, where only a total value of the sum is available and where
 iid random components are non-negative. 
 The problem is motivated by prediction problems
 in a Poisson shot noise process. %In the context,  
 %although conditional expectations given the past observation are known to be
 %best possible predictors to minimize mean squared errors, 
 %only a few special cases have been investigated with rather restricted methods. We replace
 %the prediction problem of the process with that of a random sum, which
 %is more general, and establish effective
 %numerical procedures. 
 %The methods are based on conditional
 %moments.
 In the context, although conditional moments are best possible predictors under the mean square error, 
 only a few special cases have been investigated %with rather restricted
 %methods 
 because of numerical difficulties. We replace the prediction problem of the process with that of a random sum, which
 is more general, and establish effective numerical procedures. 
 The methods are based on conditional technique together with 
 the Panjer recursion and the Fourier transform. In
 view of numerical experiments, procedures work reasonably. %, which widen the
 %application range of a random sum. 
 An application in the compound mixed Poisson
 process is also suggested. 
\end{abstract}

 % insert keywords separated by a semicolon
%\ams{}{} % insert the primary Maths Subject Classification number in the first bracket
         % and the secondary ams number(s) in the second bracket
         % e.g. \ams{60E20}{49G03;49F10}
\thanks{
Muneya Matsui's research is partly supported by the JSPS Grant-in-Aid
for Young Scientists B (25870879).
}
\maketitle 

\section{Preliminaries} 

Motivated by prediction problems in a Poisson shot noise process,
we consider two types of problems for random sums of iid
random variables (r.v. or r.v.'s for short). 
Let $N$ be a non-negative integer-valued r.v. and denote an iid sequence
of non-negative r.v's by $(X_i)_{i=1,2,\ldots}$ so that $S_N=\sum_{i=1}^N
X_i$ denotes the total sum. The distributions of both $N$ and
$X_1$ are assumed to be known. Our problem is how we could obtain the
information of the number $N$ or each component $X_i$ when we
only observe $S_N$. Although there are several methods for these
quantities such as linear predictions $cS_N$ with $c$ some constant, our methods are those by
conditional moments, which are
minimizors of the mean square error. % and which exploit structural
%property of $S_N$. 
More precisely %Hence our focus is on derivatoins
our focus is on the following two types of conditional moments: 
%We consider two types of conditional moments, one is those for
%$N^k$ given $S_N=\sum_{i=1}^N X_i$ for $k =1,2,\ldots$ namely
\begin{align}\label{def:condi:mom:rsum}
  E[N^k \mid S_N]\qquad \mathrm{and} \qquad  
%\end{align}
%for an iid sequence $(X_i)_{i=1,2,\ldots}$ taking 
%non-negative values. %in $\N_0=\{0,1,2,\ldots\}$.
%by extending the classical Panjer recursion scheme.
%The other one is those for $X_1$ given $S_N=\sum_{i=1}^N X_i$
%\begin{align}
 E[X_1^k \mid S_N],\quad \mathrm{for}\ k\in \N,
\end{align}
where $(X_i)$ may take both real and integer values %are considered for
				%$(X_i)$ 
and $\N$
denotes the set of natural numbers as usual. 

This type of random sum $S_N$ has been studied for a long time and 
has applications in a variety of fields. %As a token 
One could find many %integral-valued 
examples in the book of Feller \cite[XII]{feller:1968}
such as genetics, %animal-trapping experiments, 
required service time, 
%accumulated damage, 
cosmic ray showers, %ecology 
and automobile accidents to name just a few. 
A large number of relevant researches have been conducted, including
e.g.  calculations for probability of $S_N$ (Sundt
and Vernic \cite{sundt:vernic:2009}) or various limit theorems
(see e.g. Gut \cite{gut:2009} and consult a nice summary in Embrechts et al. \cite[2.5]{embrechts:kluppelberg:mikosch:1997}). % textbooks and papers. 
%e.g. Embrechts et al. \cite[2.5]{embrechts:kluppelberg:mikosch:1997} and large derivations.
In recent years tail asymptotics have intensively studied, since
accurate calculations of tail probabilities of $S_N$ are
computationally quite expensive, while they are required in
applications. See Jessen and Mikosch \cite{jessen:mikosch:2006} for a
survey with regularly varying tails and Goldie and Kl\"upperberg
\cite{goldie:kluppelberg:1998} for that with subexponential tails. 

In this paper we do not go further into asymptotics but investigate precise calculations of
quantities \eqref{def:condi:mom:rsum}, which have not been studied yet
except for some special cases (see Subsection \ref{motivatoin}). 
%In what follows, we explain a motivating example and why
%\eqref{def:condi:mom:rsum} appears in applications together with limitatins of preceding researches. 
%{\bf Motivation}
%In our paper, 
We rely on two numerical methods, i.e. the Panjer recursion and the
Fourier method, which are useful tools for computing $P(S_N=n)$
and which are competitive (\cite{embrechts:frei:2009}). 
%Concerning the exact probability $p_n=P(S_N=n)$, 
The Panjer recursion scheme originated in Panjer \cite{panjer:1981} 
%e.g. \cite{RSST:1999,mikosch:2009} 
is known to be stable when $N$ belongs to the Panjer class in most cases
(\cite{panjer:wang:1993}). Meanwhile, the Fourier method could be
applicable to general $N$, though it requires accurate numerical
integrals. Here we show that
these methods could be useful tools for computing 
quantities in \eqref{def:condi:mom:rsum} and establish efficient
numerical procedures. Our methods do not depend on specific 
distributions on $N$ and $X_1$ and therefore could be applicable under general settings. 
%conditional moments of the generating number or elements of
%the random sum given the value of the whole sum.
%except for a binomial distribution \cite{panjer:wang:1993}.

In the remainder of this section, we present a motivating application 
and its literature in Subsection \ref{motivatoin} and introduce
notations used in Subsection \ref{notations}. 
In Section
\ref{sec:rsodv} random sums of discrete r.v.'s are treated% in 
, where computations of $E[N^k \mid
S_N],\,k\in \N$ are investigated in Subsection \ref{sec:discrete:number}
and those of %$E[X_1^k \mid S_{N+1}]$ and 
$E[X_1^k\mid S_{N}]$ are studied in
Subsection \ref{sec:discrete:element}. 
Both the recursion method and the Fourier method are investigated.
% for $E[N^k \mid
%S_N],\,k\in \N$. However, only the Fourier method is applied for 
%Since %we have difficulty in applying 
%the recursion for $E[X_1^k\mid
%S_{N}]$ seems unfeasible, we alternatively consider $E[X_1^k \mid S_{N+1}]$ to which the
%recursion is applicable. 
%and $E[X_1^k\mid S_{N}]$, and the recursion is
%given only for $E[X_1^k \mid S_{N+1}]$. 
In Section \ref{RSOCRV}, we consider
random sums of non-negative continuous r.v.'s, %Although we derive
%integral equations for quantities $E[N^k \mid S_N \le x]$ for the Panjer
%class $N$, a direct application
%seems intractable. Alternatively, we resort mainly to 
where we take the Fourier
approach for computations of both $E[N^k \mid S_N \le x]$ and
$E[X_1^k \mid S_N\le x]$,\,$x>0$. 
Finally in Section \ref{sec:numeraical:ex}, numerical examples are
given, which show that proposed methods work reasonably. As applications, we
consider predictors for both the Poisson shot noise process and the
compound mixed Poisson process.

% for $E[N^k \mid
%S_N],\,k\in \N$, $E[X_1^k \mid S_{N+1}]$ and $E[X_1^k\mid S_{N}]$.
%Numerical examples for predictors of Poisson cluster models are 
%also presented for compound Poisson clusters. 

\subsection{Motivating application and its literature}\label{motivatoin}
A motivating example is prediction in the Poisson shot noise process of the form
%a Poisson cluster model 
%treated in Matsui and Mikosch \cite{matsui:mikosch:2010} which has the form 
%(a shot noise process) which has a plenty of applications
%in rather different areas (see \cite{bartlett:1963}, \cite{verejones:1970}
%\cite{neyman:scott:1958} and \cite{lewis:1964}).
%Especially, quantities \eqref{def:condi:mom:rsum} appear in prediction
%problems.  %of Poisson cluster models. 
%of loss reserves in the non-life insurance (see Jessen et
%al. \cite{jessen:mikosch:samorodnitsky:2009}, Matsui and Mikosch \cite{matsui:mikosch:2010} ). 
%Recently, the calculations for $E[N^k \mid S_N]$ 
%(especially for $E[N\mid S_N]$ and $\mathrm{Var}(N\mid S_N)$)
%are requited in
%modeling of loss reserves; see,
%Jessen et al. \cite{jessen:mikosch:samorodnitsky:2009}, Matsui and
%Mikosch \cite{matsui:mikosch:2010}
%and see also the older papers Norberg \cite{norberg:1993,norberg:1999}. We take
%examples and explain why these calculations appear. 
%Our main motivation is on $E[N\mid S_N]$ and $\mathrm{Var}(N\mid S_N)$ which
%arise from the prediction problems in 
%Consider a kind of shot noise process
%called a Poisson cluster model defined in e.g. \cite{matsui:mikosch:2010}. 
%The model
\begin{align}\label{def:posnp}
 M(t) = \sum_{i=1}^{N(t)} L_i(t-T_i),\quad t>0,
\end{align}
where $0<T_1<T_2<\cdots$ are points of a homogeneous Poisson $N(t)$ with
intensity $\lambda>0$ and $(L_j)$ is a sequence of iid L\'evy 
processes independent of $(T_i)$ and such that $L_i(t)=0\ a.s.\,t\le
0$. The process of this type has
many applications in rather different areas (see \cite{bartlett:1963}, \cite{verejones:1970}
\cite{neyman:scott:1958} and \cite{lewis:1964}). 
One of important research topics is the prediction of future increments 
 $M(t,t+s]:=M(t+s)-M(t),\,s,t>0$ based on the present observation $M(t)$. 
For example, in non-life insurance $M(t,t+s]$ is interpreted as the
number or amount of future payment in the interval
$(t,t+s]$ from an insurance company to the insured.  
Another interpretation is that $M(t,t+s]$ may describe the workload to
be managed by a large computer network for sources in the interval
$(t,t+s]$. 
Due to the properties of both L\'evy and Poisson processes, the
prediction of future increments $M(t,t+s]$ given
$M(t)$ reduces to 
\begin{align}\label{predi:posnp}
 E[M(t,t+s] \mid M(t)] %&= E\Big[\sum_{j=N(t)}^{N(t+s)}L_j(t+s-T_j)\mid
 %M(t)\Big] + E\Big[\sum_{j=1}^{N(t)}L_j(t-T_j,t+s-T_j]\mid M(t)\Big] \\
 %&= E[N(t,t+s]]E[L(t+s-V)] + E\Big[\sum_{j=1}^{N(t)}L_j(s) \mid M(t)
 %\Big] \\
 %& 
 = E[N(s)]E[ L_1(t+s-U)] + E[L_1(s)] E[N(t)\mid M(t)],
\end{align}
where $U$ is a uniform r.v. on $(0,t)$ denoted by $U(0,t)$ independent
of $(L_i)$. The proof of \eqref{predi:posnp} is given in Appendix
\ref{append:sec} or \cite[(2.1)]{matsui:mikosch:2010}. 
Here computations of $E[N(s)], E[L_1(s)]$ and $E[L_1(t+s-U)]$ are trivial. 
Since points $(T_i)$ of 
Poisson have the order statistic property, we can regard the sequence
$(T_i)$ in the quantity
$E[N(t)\mid M(t)]$ as that of iid $U(0,t)$ r.v.'s. Accordingly,  
taking $X_i:=L_i(t-T_i)$ and
$N:=N(t)$ of $M(t)$, we obtain the form \eqref{def:condi:mom:rsum}. Similarly, higher conditional moments
$E[M(t,t+s]^k \mid M(t)],\,k\in \N $ are obtained as functions of $E[N^k(t)\mid M(t)]$. 

A series of papers \cite{mikosch:2009}, \cite{jessen:mikosch:samorodnitsky:2009} and
\cite{matsui:mikosch:2010} assumes particular marginal distributions for
$(L_i)$ such as Poisson or negative binomial, and exploits their specific 
properties to obtain the conditional moments. Although asymptotic
behaviors of $E[N\mid S_N=k],\,k \to \infty$ have been studied in
\cite{jessen:mikosch:samorodnitsky:2009} and \cite{rolski:tomanek:2011,rolski:tomanek:2013}, only limited distributions
are treated. %Accordingly more general results are required in terms of
%practical applications. 
Our methods presented here require no
particular assumptions on distributions of $N$ and $X_1$ and therefore could be applicable under
more general settings than those of previous papers. 

%In the later section, we give another application
%by prediction for the compound Pascal process. 

%, and other related papers ,
%properties of a particular class of distributions, i.e., Poisson or negative
%binomial were exploited to obtain conditional moments. Asymptotic
%behavior of conditional moments have also studied in Rolski and Tomanek 
%\cite{rolski:tomanek:2011,rolski:tomanek:2013}. However, they are still
%limited.  
%We give another aplication of prediction which is by compound Pascal processes.

%{\bf Explanation of the Panjer recursion of $p_n=P(S_N=n)$.}\\

%provides an efficient numerical
%methods. Although there are several alternative methods such as Fourier
%inversions or Mote Carlo procedures, for the Panjer class distributions
%exept for a binomial distribution,
%this recursion scheme is known to be surely stable \cite{panjer:wang:1993}. 
%In this paper we show that the Panjer recursion is also useful
%to calculate conditional moments of the generating number or elements of
%the random sum given the value of the whole sum.

\subsection{Necessary notations and tools}\label{notations}
%For the latter convenience, we define 
Throughout we use the following notations related
with generating functions. 
 For a fixed r.v. $X$ and a non-negative
function $f$ and $|u|\le 1$,  
\begin{align*}
 G_X(u) := E[u^X],\qquad G_f(u) := \sum_{k=0}^\infty u^k f(k)
%\; (discrete)
,\qquad
 G_{df}(u) := \int u^x df(x), %\; (continuous),
\end{align*}
where the last one is defined as a Riemann-Stieltjes integral if
exists. 
From these quantities we can obtain the Fourier (-Stieltjes) transforms
$\phi_{\{\cdot\}}(u) = G_{\{\cdot\}}(e^{iu}).$ Note that we use generating functions not
only for r.v.'s but also for discrete sequences (see \cite{wilf:1993}),
though after a proper standardization, they are the same. 
 We write
$\N_0:=\{0,1,2,\ldots\}$ and $\R_+:=[0,\infty)$ in the sequel. Moreover, braces
$
  \left\{ 
\begin{smallmatrix}
n\\k
\end{smallmatrix}
\right\}
$
denote the Stirling numbers of the second kind (see
\cite[p.824]{abramowitz:stegun:1972}): the number of ways of
partitioning a set of $n$ elements into $k$ non-empty sbsets. 

We say that the probability mass function $q_n=P(N=n)$ belongs to the Panjer $(a,b)$
class if it satisfies
\[
q_n=\Big(a+\frac{b}{n}\Big)\,q_{n-1},\,\quad n\in \N. 
\]
for $a+b\ge 0$ and $a<1$ (\cite[p.122]{RSST:1999,mikosch:2009}).
%some $a, b\in \mathbb{R}$. 
Poisson, negative binomial and binomial
distributions belong to this class. 
For later use, we present the Panjer recursion formula (see
\cite{RSST:1999,mikosch:2009} for details and
the proof). 
%We omit the proof since one can find it easily in textbooks such as \cite{RSST:1999,mikosch:2009}. 
\begin{theorem}\label{thm:basic}
 Suppose that $N$ belongs to the Panjer $(a,b)$ class and denote an iid
 sequence of non-negative integer-valued r.v.'s by $(X_i)$. Then 
 \begin{align*}
  P(S_N=0) = E[P(X_1=0)^N],\hspace{6.2cm} & n=0, \\
  P(S_N=n) = \frac{1}{1-aP(X_1=0)} \sum_{j=1}^n
  \Big(a+\frac{bj}{n}\Big)P(X_1=j)P(S_{N}=n-j),\quad & n\ge 1.
 \end{align*}
 Here we let $0^0=1$ conventionally. 
\end{theorem}

%\begin{remark}
% Our focusing Panjer class, $\phi_{S_N}=E[\phi_{X_1}^N]$ are 
% \begin{align*}
%  \phi_{S_N}(u)= \big( p \phi_{X_1}(u) +(1-p)\big)^n, &\hspace{1cm}
%  \mathrm{Binomial} \\
%   \phi_{S_N}(u)= e^{\lambda(\phi_{X_1}(u)-1)}, &\hspace{1cm}
%  \mathrm{Poisson} \\
%   \Phi_{S_N}(u) = \left(\frac{p}{1-\phi_{X_1}(u)q}\right)^r,
%  &\hspace{1cm}\mathrm{Negative\ binomial}
% \end{align*} 
%\end{remark}

\section{Random sums of discrete random variables}
\label{sec:rsodv}

\subsection{Estimation of number of iid components}
\label{sec:discrete:number}
In this section calculations for conditional moments $E[N^k \mid
S_N],\,k\in \N$ will be investigated, where iid random components $(X_i)$
are integer-valued. In case %the random number 
$N$ belongs to the Panjer
class, we apply the recursion formula to the calculation of  $E[N^k \mid
S_N]$. For general %non-negative integer r.v. 
$N$, we consider the
generating function of $E[N^k \mid S_N=\cdot]$ and then apply the
inversion formula. 

%In this section conditional moments $E[N^k \mid S_N]$ will be 
%investigated. Firstly we directly apply the Panger recursion, and 
%then consider the generating function of $E[N^k \mid S_N=\cdot]$ for
%general non negative integer r.v. $N$. Fourier inversion method and 
%derivative method are suggested to obtain $E[N^k \mid S_N=\ell]$
%from the generating function.  

Throughout we denote the expectation of a r.v. $X$ over a measurable subset $A \subset \Omega$ by 
$ E[X;A] = E [X I_{ ( X \in A ) } ]$ following Kallenberg \cite[p.49]{kallenberg:2002}. 
%For the calculation of $P(N \mid S_N)$, since several methods have already
%been proposed, we only give results without proof and 
Since we obtain $P(S_N)$ by Theorem \ref{thm:basic}, we 
mainly consider
$E[N^k; S_N],\,k\in\N$ which yields $E[N^k \mid S_N]=E[N^k;S_N]/P(S_N)$.

\begin{theorem}
\label{thm:reccursion-n}
 Let $N$ belong to the Panjer $(a,b)$ class and iid r.v.'s $(X_i)$ take values
 in $\N_0$. Let $C_0:=aP(X_1=0)$. Then, the restricted moments
 $m_k(\ell) :=E[N^k ;S_N=\ell],\,k,\,\ell \in \N_0$ satisfy the
 recursion, 
\begin{align}
\begin{split}
 m_k(0) &= E[N^k P(X_1=0)^N],  \\
 m_k(\ell) & = \frac{1}{1-C_0%aP(X_1=0)
} \Big\{
C_0%aP(X_1=0) 
\sum_{j=0}^{k-1} \binom{k}{j}\, m_j(\ell) + \sum_{j=1}^\ell \Big(
a+\frac{bj}{\ell}
\Big) P(X_1=j)\sum_{i=0}^k \binom{k}{i}\, m_i(\ell-j)
\Big\},\quad \ell \ge 1,\label{eq:reccursion1-1}
\end{split}
\end{align}
where $m_0(\ell)=P(S_N=\ell)$. 
\end{theorem}
Since for the calculation of $m_k(\ell)$, a combination of 
$m_i(\ell),\,i\le k-1$ and
$m_i(j),\,i\le k,\, j\le \ell-1$ is sufficient, we can recursively
calculate the quantity. 

\begin{proof}
 The iidness of $(X_i)$ and independence between $N$ and $(X_i)$
 yield 
 \begin{align*}
  m_k(0) =E[N^k ;S_N=0] = E[N^k E[I_{(S_N=0)}\mid N]] = E[N^k P(X_1=0)^N]. 
%E[N^k ;S_N=0,\,N=0]+ E[N^k;
%  S_N=0,N>0] \\
% &= E[N^k;S_N=0,N>0] \\
% &= E[N^k P(X_1=0)^N]. 
 \end{align*}
Next we consider $m_k(\ell),\,\ell \ge 1$. Conditioning
 argument and the Panjer $(a,b)$ class assumption yield  
\begin{align}
 m_k(\ell) &= \sum_{i=1}^\infty i^k P(S_i=\ell) q_i %\nonumber \\&
 = \sum_{i=1}^\infty P(S_i=\ell) i^k \Big(
a+\frac{b}{i}
\Big)\,q_{i-1} \label{pf:condi-moment0}
\end{align}
where $q_n=P(N=n),\,n\in \N_0$. Since $(X_j)$ are iid
 r.v.'s 
\begin{align*}
 \Big(a+\frac{b}{i}\Big) &= a+b \frac{1}{i}\sum_{j=1}^i E\Big[
 \frac{X_j}{S_i} \mid S_i=\ell
\Big] %\\&
= a+bE\Big[\frac{X_1}{S_i} \mid S_i=\ell \Big] %\\&
= E\Big[
a+b\frac{X_1}{\ell} \mid S_i=\ell
\Big].
\end{align*}
Moreover,
\begin{align*}
 E\Big[a+\frac{bX_1}{\ell} \mid S_i=\ell \Big] &= \sum_{j=0}^\ell \Big(
a+\frac{bj}{\ell}
\Big) P(X_1=j \mid S_i=\ell) %\nonumber \\
%&= \sum_{j=0}^\ell \Big(
%a+\frac{bj}{\ell} \Big)
% \frac{P(X_1=j,\,S_i-X_1=\ell-j)}{P(S_i=\ell)}\nonumber \\
%&
= \sum_{j=0}^\ell \Big(
a+\frac{bj}{\ell} \Big)
 \frac{P(X_1=j)P(S_{i-1}=\ell-j)}{P(S_i=\ell)}. %\label{pf:condi-moment1}
\end{align*}
Substitution of this %\eqref{pf:condi-moment1} 
into $(a+b/i)$ of
 \eqref{pf:condi-moment0} and multiple interchanges of the order of
 summations give
\begin{align*}
 m_k(\ell) & = \sum_{i=1}^\infty \sum_{j=0}^\ell 
\Big(a +\frac{bj}{\ell}\Big) 
P(X_1=j) P(S_{i-1}=\ell-j) i^k q_{i-1} \\
&= \sum_{j=0}^\ell \Big(
a +\frac{bj}{\ell}\Big) P(X_1=j)\, \Big\{
\sum_{i=1}^\infty P(S_{i-1}=\ell -j) i^k q_{i-1}
\Big\} \\
%&= \sum_{j=0}^\ell \Big(
%a +\frac{bj}{\ell}\Big) P(X_1=j)\, \Big\{
%\sum_{i=0}^\infty P(S_{i}=\ell -j) (i+1)^k q_{i}
%\Big\} \\
&= \sum_{j=0}^\ell  \Big(
a +\frac{bj}{\ell}\Big)
P(X_1=j) \sum_{i=0}^\infty P(S_i=\ell-j)
\sum_{h=0}^k \binom{k}{h}\,%{}_kC_h
 i^h q_i \\
%&= \sum_{j=0}^\ell  \Big(
%a +\frac{bj}{\ell}\Big)
%P(X_1=j)\sum_{h=0}^k \binom{k}{h}\, %{}_kC_h 
%\Big[
%\sum_{i=0}^\infty P(S_i=\ell-j)i^h q_i
%\Big]\\
&= \sum_{j=0}^\ell  \Big(
a +\frac{bj}{\ell}\Big)
P(X_1=j) \sum_{h=0}^k \binom{k}{h} %{}_kC_h
\, m_h(\ell-j) \\
%&= aP(X_1=0) \sum_{h=0}^k \binom{k}{h}\, %{}_k C_h 
%m_h(\ell) + \sum_{j=1}^\ell  \Big(
%a +\frac{bj}{\ell}\Big) P(X_1=j) \sum_{h=0}^k \binom{k}{h}\, %{}_kC_h 
%m_h(\ell-j) \\
&= a P(X_1=0)m_k(\ell) + aP(X_1=0) \sum_{h=0}^{k-1}\binom{k}{h}\, %{}_k
 %C_h 
m_h(\ell)
 +\sum_{j=1}^\ell \Big(
a+\frac{bj}{\ell}
\Big) P(X_1=j) \sum_{h=0}^k \binom{k}{h}\, %{}_k C_h 
m_h(\ell-j). 
\end{align*}
Thus we obtain the desired result. 
\end{proof}

%\begin{corollary}
% In Theorem \ref{thm:reccursion-n} take $k=1$ and we obtain 
%\begin{align*}
% m_1(0) &= E[NP(X_1=0)^N] \\
% m_1(\ell) &= \frac{1}{1-aP(X_1=0)} \Big\{
% P(S_N=\ell) + \sum_{j=1}^\ell \Big(
% a+\frac{bj}{\ell}\Big)P(X_1=j) m_1(\ell-j)
%\Big\}
%\end{align*}
%\end{corollary}

%\begin{proof}
% Take $k=1$ in \eqref{eq:reccursion1-1} and then
% \begin{align*}
%  m_1(\ell) &= \frac{1}{1-aP(X_1=0)} \Big\{
% aP(X_1=0) {}_1C_0 m_0(\ell) + \sum_{j=1}^\ell P(X_1=j) \Big(
% {}_1C_0 m_0(\ell-j) + {}_1C_1 m_i(\ell-j)
%\Big)
%\Big\} \\
% & = \frac{1}{1-aP(X_1=0)} 
% \Big\{
% aP(X_1=0)P(S_N=\ell) + 
% \sum_{j=1}^\ell \Big( a+\frac{bj}{\ell} \Big) 
% P(X_1=j)
% \Big(
% P(S_N= \ell-j)+m_1(\ell-j)
% \Big)
%\Big\} \\
% & = \frac{1}{1-aP(X_1=0)} \Big\{
% aP(X_1=0) P(S_N=\ell)+ (1-aP(X_1=0))P(S_N=\ell)
%  \\
% & \hspace{3cm}+  \sum_{j=1}^\ell \Big( a+\frac{bj}{\ell} \Big) 
% P(X_1=j)
% \Big(
% P(S_N= \ell-j)+m_1(\ell-j)
% \Big) \Big\},
% \end{align*}
% which concludes the result. 
%\end{proof}

If we take $k=1$ with $\ell \ge 1$
in Theorem \ref{thm:reccursion-n}, a rather simple expression is
obtained 
%take $k=1$ and $\ell \ge 1$ to
%obtain the conditional expectation,
\[
 m_1(\ell) = \frac{1}{1-C_0} \Big\{
C_0 P(S_N=0)
+ \sum_{j=1}^\ell \Big(a+\frac{bj}{\ell}\Big)P(X_1=j)\big(
P(S_N=\ell-j)+m_1(\ell-j)
\big)\Big\},
\]
which together with Theorem \ref{thm:basic}, yields the conditional expectation. 

Next we consider the generating function for $m_k(\ell)$ with $N$
a general r.v. 
\begin{proposition}\label{thm:fourier:inversion}
 Let $N$ be a r.v. on $\N_0$ and let $(X_i)$ be an iid sequence of
 r.v.'s on $\N_0$. Assume $EN^k<\infty,\,k\in \N.$ Then the generating
 function of
 the truncated $k$-th moment $m_k(\ell)=E[N^k;S_N=\ell]$ has the form
 \begin{align}
\label{eq:second:recursion}
  G_{m_k}(u) &= 
 \sum_{j=1}^k 
 \left\{ 
\begin{array}{c}
      k  \\
      j   
    \end{array}
\right\} G_{X_1}^{j}(u)\, G_{N}^{(j)}(G_{X_1}(u)),\quad |u|\le 1,
 \end{align}
where the quantities by braces $\{\}$ denote the Striling number of the
 second kind, and
 $G^{(j)}_Y(u),\,j\in \N$ denotes the $j$-th derivative of $G_Y(x)$ at $x=u$. 
%\footnote{
%In the sequel, we use this notation for the Striling number of the second 
%kind in without mentioning them. 
%} \cite[p.824]{abramowitz:stegun:1972}, and
 $G^{(j)}_Y(u),\,j\in \N$ denotes the $j$-th derivative of $G_Y(x)$ at $x=u$. 
\end{proposition}

\begin{proof}
 A direct calculation yields 
 \begin{align*}
  G_{m_k}(u) &= \sum_{\ell=0}^\infty u^\ell E[N^k;S_N=\ell] %\\
             %&= E \sum_{\ell =0}^\infty u^\ell N^k P(S_N=\ell \mid N) \\
             %&
= EN^k \sum_{\ell=0}^\infty u^\ell P(S_N=\ell \mid N) 
%\\&
= EN^k G_{X_1}^N(u),
 \end{align*}
 where $EN^k< \infty$ assures Fubini's theorem since $ |G_{X_1}(u)| \le 1$. We use the relation of the
 falling factorial $(x)_k=x(x-1)\cdots(x-k+1)$ and $x^k$,
 \begin{align}
\label{relation-falling-factorial}
  \sum_{j=1}^k 
 \left\{ 
\begin{array}{c}
      k  \\
      j   
    \end{array}
\right\} (x)_j =x^k,\qquad \left\{ 
\begin{array}{c}
      k  \\
      0   
    \end{array}
\right\}=0,\quad k>0,
 \end{align}
namely,
\begin{align*}
 E[N^k G_{X_1}^N(u)] &= \sum_{j=1}^k 
 \left\{ 
\begin{array}{c}
      k  \\
      j   
    \end{array}
\right\} E[(N)_j\, G_{X_1}^N(u)] %\\
%&= \sum_{j=1}^k 
% \left\{ 
%\begin{array}{c}
%      k  \\
%      j   
%    \end{array}
%\right\}
%G_{X_1}^j(u)E[(N)_jG_{X_1}^{N-j}(u)] \\
%&
= 
\sum_{j=1}^k 
 \left\{ 
\begin{array}{c}
      k  \\
      j   
    \end{array}
\right\} G_{X_1}^j(u) E[(N)_jG_{X_1}^{N-j}(u)]
 =\eqref{eq:second:recursion}, 
%G_{X_1}^j (u)\, G^{(j)}_N\big(G_{X_1}(u)\big), 
\end{align*}
where we change the order of derivatives and the summation, which is valid
 from $EN^k <\infty$ and $|G_{X_1}(u)|\le 1$. 
\end{proof}

%In the sequel, we use notations for the Striling number of the second 
%kind in 
%Proposition \ref{thm:fourier:inversion} without mentioning them. 

In order to obtain $m_k(\ell)=E[N^k;S_N=\ell]$ from $G_{m_k}$, two 
methods are considered. One requires numerical integrations and the
other needs derivatives of $G_{m_k}$ at the origin. 
%is applicable to a general non-negative
%r.v. $N$, but requires numeraical integrations. 
Since $|G_{m_k}(e^{iv})|^2\le (EN^k)^2<\infty$, %which implies
we have $G_{m_k}(e^{iv}) \in L^2(-\pi,\pi)$. 
Then the Fourier expansion of $G_{m_k}(e^{iv})$ is guaranteed and their
coefficients satisfy 
%for a general r.v. $N$, that is to apply the inversion
formula 
\begin{align}
\label{eq:Fourier:inversion}
 m_k(\ell)= \frac{1}{2\pi}\int_{-\pi}^\pi e^{-i\ell u} G_{m_k}(e^{iu})\d
 u,\quad \ell\in \N_0, 
\end{align}
which correspond to the inversion of the Fourier transform
$G_{m_k}(e^{iv})$. On the other hand, if we take derivatives of $G_{m_k}$ at the origin, we
obtain 
$
 m_k(\ell) =\frac{1}{\ell!} G_{m_k}^{(\ell)}(0).
$
In view of \eqref{eq:second:recursion}, however, the calculation of
$G_{m_k}^{(\ell)}$ would yield additional complexities, though we may possibly
find some efficient recursion methods for a limited class of $N$. The
choice of the two methods depends on distributional assumptions on $N$
and $X_1$ and we need numerical experiments to judge which is
better.

\subsection{Estimation of magnitude of each iid component}
\label{sec:discrete:element}
In this subsection we consider the expected magnitude of
r.v. $X_1^k,\,k\in\N$ under the observation of the total number $S_N$. Since
the conditional moments minimize mean squared errors, we will consider
$\chi_k=E[X_1^k\mid S_N],\,k\in \N$. Since the direct application of the
Panjer recursion seems difficult for $\chi_k$ and easy for
$\chi_{k+}:=E[X_1^k \mid S_{N+1}]$, we derive the recursion only for
$\chi_{k+}$. Meanwhile, the Fourier approach is applied to both.
% for $\chi_k$, which is
%also applied to $\chi_{k+}$. 

%In this subsection we consider the expected magnitude of 
%r.v. $X_1$ under
%observation of the total number $S_N$. Since the conditional moment minimizes
%mean squared error, we will consider $E[X_1\mid S_{N}]$. 
% Since the direct application of the Panjer recursion seems difficult for 
%$E[X_1\mid S_{N}]$ and easy for $E[X_1\mid S_{N+1}]$, we only derive the 
%reccursion for $E[X_1\mid S_{N+1}]$ and 
%consider Fourier expantions for 
%$E[X_1\mid S_N]$. 

\begin{theorem}\label{thm:element:fourier}
 Let $N$ be a Panjer $(a,b)$ class distribution and let $(X_i)$ be a
 sequence of iid r.v.'s on $\N_0$. % such that $X_1 \in \N_0$. 
 Assume $EX_1^k<\infty,\,k\in \N$, then the truncated $k$-th moment $\chi_{k+}(\ell) =
 E[X_1^k;S_{N+1}=\ell],\,\ell\in \N$ has the form 
\begin{align*}
 \chi_{k+}(1) &= P(X_1=1)E[P(X_1=0)^N],\quad \mathrm{and\quad for}\ \ell \ge 2, \\
 \chi_{k+}(\ell) &= \ell^k P(X_1=\ell )P(S_N=0) + \frac{1}{1-a P(X_1=0)} 
 \sum_{j=1}^{\ell-1} P(X_1=j)\, \Big\{
 a \chi_{k+}(\ell-j) + \frac{bj^k}{\ell-j}\chi_{1+}(\ell-j) 
 \Big\}.
\end{align*}
\end{theorem}

\begin{proof}
 %Trivially $\chi_{k+}(0)=0$. 
 For $\ell=1$, due to the iidness of $(X_i)$,
\begin{align*}
 \chi_{k+}(1) &= E[X_1^k ; S_{N+1}=1] %\\
% &= E[X_1^k;S_{N+1}=1,\,N=0] + E[X_1^k; S_{N+1}=1,\,N\ge 1] \\
% &= E[X_1^k; X_1=1,\,N=0] + E[X_1^k ;X_1=1,X_2=0,\ldots,X_{N+1}=0,\,N\ge
% 1] \\
% &= P(X_1=1,N=0)+P(X_1=1,X_2=0,\ldots,X_{N+1}=0,\,N\ge 1) \\
% &= P(X_1=1)P(N=0) + E[P(X_1)P(X_1=0)^N; N\ge 1] \\
% &
= P(X_1=1) E[P(X_1=0)^N].
\end{align*}
 Let $C_1:= 1/(1-aP(X_1=0))$. For $\ell \ge 2$, the property of $N$ yields 
\begin{align*}
 \chi_{k+}(\ell) &= E[X_1^k; S_{N+1}=\ell] \\
 &= \sum_{j=1}^\ell j^k P(X_1=j)P(S_N=\ell-j) \\
 &= \ell^k P(X_1=\ell) P(S_N=0) + \sum_{j=1}^{\ell-1} j^k
 P(X_1=j)C_1 %\frac{1}{1-aP(X_1=0)} 
 \sum_{m=1}^{\ell-j} 
 \Big(
 a+\frac{bm}{\ell-j}\Big) P(X_1=m)P(S_N=\ell-j-m) \\ %q_{\ell-j-m} \\
 &= \ell^k P(X_1=\ell) P(S_N=0) + C_1 %\frac{1}{1-aP(X_1=0)} 
 \Bigg\{
 a \sum_{j=1}^{\ell-1} \sum_{m=1}^{\ell-j}j^k P(X_1=j) P(X_1=m)
 P(S_N=\ell-j-m) \\
 & \hspace{5cm}
 +b\sum_{j=1}^{\ell-1} \frac{j^k}{\ell-j}P(X_1=j)
 \sum_{m=1}^{\ell-j}m P(X_1=m) P(S_N=\ell-j-m)
%E[X_1;S_{N+1}=\ell-j]
\Bigg\} \\
 &= \ell^k P(X_1=\ell)P(S_N=0) + C_1 %\frac{1}{1-aP(X_1=0)} 
 \Bigg\{
 a\sum_{m=1}^{\ell-1} P(X_1=m)E[X_1^k; S_{N+1}=\ell-m] \\
 & \hspace{5cm}
 + b \sum_{j=1}^{\ell-1}\frac{j^k}{\ell-j} P(X_1=j) E[X_1;
 S_{N+1}=\ell-j] \Bigg\}, %\\
%  &= \ell^k P(X_1=\ell)P(S_N=0) + \frac{1}{1-aP(X_1=0)} \Big\{
% a\sum_{j=1}^{\ell-1} P(X_1=j)\chi_{k+}(\ell-j) b \sum_{j=1}^{\ell-1}\frac{j^k}{\ell-j} P(X_1=j)\chi_{1+}(\ell-j) \Big\} 
\end{align*}
where in the third step, we use the Panjer recursion for 
 $P(S_N=\ell-j)$. 
 Finally, we arrange two sums and obtain the result. 
\end{proof}

For the calculation of $\chi_k=E[X_1^k\mid S_N]$ 
%For computation of $E[X_1\mid S_N]$, 
a direct application of the
Panjer recursion seems difficult and alternatively we try the Fourier
methods. For this we need the generating function of $\chi_k$. 

%We give alternative methods which are
%applicable for $E[X_1\mid S_N]$ with a general non-negative r.v. $N$ and which are 
%based on the generating function. 

\begin{proposition}
 Let $N$ be a r.v. on $\N_0$ and let $(X_i)$ be an 
 iid sequence of r.v.'s on $\N_0$. Assume $E X_1^k < \infty,\,k\in\N$,
 then the generating function of the truncated $k$-th moment $\chi_k( \cdot )=E[X_1^k; S_N=\cdot ]$ has
 the form 
\begin{align}\label{def:dft:chi_i}
 G_{\chi_k}(u) &= \sum_{j=1}^k 
 \left\{ 
\begin{array}{c}
      k  \\
      j   
    \end{array}
\right\} u^j G_{X_1}^{(j)}(u)\, \frac{G_{S_N}(u)}{G_{X_1}(u)},\quad |u|\le
 1.
\end{align}
%where braces $\{\}$ mean the striling numbers of the second kind. 
\end{proposition}

\begin{proof}
In view of 
\begin{align*}
 \chi_k(\ell) = E[X_1^k; S_N=\ell] = \sum_{j=1}^\ell j^k P(X_1=j)P(S_{N-1}=\ell-j),   
\end{align*}
the function $\chi_k$ is the convolution of two non-negative functions
 $g_1(j):=j^kP(X_1=j)$ and $g_2(j):=P(S_{N-1}=j)$. 
 %Hence it suffices to obtain their
 %generating functions respectively. 
 %the
 %generating function $\chi_k$ is the product of those for $f$ and $g$. 
 Since $G_{S_{N-1}}(u)=\sum_{j=1}^\ell u^j P(S_{N-1}=j)=E[G_{X_1}^{N-1}(u)]=G_{S_N}(u)/G_{X_1}(u)$, 
 the generating function of $g_1(j)$ % := j^kP(X_1=j),\,1\le j\le \ell$ 
 is enough. 
 %it suffices to obtain that for $f$. %$f(i)=i^kP(X_1=i)$. 
 We use the relation of the falling factorial
 \eqref{relation-falling-factorial} and obtain 
%$(x)_j=x(x-1)\ldots (x-k+1)$ and $x^k$,
% \[
%  \sum_{j=1}^k  \left\{ 
%\begin{array}{c}
%      k  \\
%      j   
%    \end{array}
%\right\} (x)_j =x^k,
% \]
%namely
\begin{align*}
 E[X_1^ku^{X_1}] = \sum_{j=1}^k   \left\{ 
\begin{array}{c}
      k  \\
      j   
    \end{array}
\right\} E [(X_1)_j u^{X_1}] 
%& =  \sum_{j=1}^k  \left\{ 
%\begin{array}{c}
%      k  \\
%      j   
%    \end{array}
%\right\} u^j E[(X)_ju^{X_1-j}] \\
 =  \sum_{j=1}^k  \left\{ 
\begin{array}{c}
      k  \\
      j   
    \end{array}
\right\} u^j G_{X_1}^{(j)}(u),
\end{align*} 
where we apply Fubini's theorem, which is possible by $EX_1^k<\infty$. %Hence we obtain the result. 
Now the product of $G_{S_{N-1}}$ and $E[X_1^ku^{X_1}]$ yields the result.
\end{proof}

Similarly as before, two methods are considered to obtain $\chi_k( \cdot
)$ from $G_{\chi_k}$. One is to use derivatives at the origin, 
%\[
$ \chi_k(\ell) = G^{(\ell)}_{\chi_k}(0)/ \ell !, \ell \ge 1$.
%\]
The other is the inversion of generating function
\[
 \chi_k(\ell) = \frac{1}{2\pi} \int_{-\pi}^\pi e^{-i\ell u} G_{\chi_k}(e^{iu})\d
 u,\quad \ell \in N_0.  
\]
In view of 
\eqref{def:dft:chi_i}, the former method requires some efficient algorithm for
calculating derivatives of $G_{\chi_k}$, whereas for the second one, accurate numerical integrations are inevitable.

\section{Random sums of continuous random variables}
\label{RSOCRV}
In this section, we assume continuous distributions for 
an iid random sequence $(X_i)$ taking values on $\R_+$, while keeping
$N$ to be r.v. on $\N_0$. Similarly
as before we consider $E[N^k\mid S_N]$ and $E[X_1^k\mid S_N]$, $k\in
\N$. Here the Fourier Stieltjes transform (FST for short) is our main tool. 
%In
%the continuous distribution case, the Fourier transform is our main tool.

\subsection{Estimation of random number from random sum}
We firstly consider $E[N^k \mid S_N \in [0,x]]=E[N^k \mid S_N \le x]$ for $x\in
\R_+$ and $k\in \N$. We are starting to observe the integral equation as
in \cite[Sec. 4.4.3]{RSST:1999}, 
which corresponds to the recursion formula when $X_1$ is a discrete
distribution.  

\begin{theorem}\label{thm:conti:panjer}
 Let $N$ be a Panjer $(a,b)$ class distribution and assume iid r.v.'s
 $(X_i)$ take values on $\R_+$ with common distribution $F_{X_1}$. Then the
 restricted $k$-th moment to the Borel set by $\{S_N \le x\}$, $m_k(x)=E[N^k;S_N \le x],\,k\in\N$ satisfies
 the integral equation,
\begin{align}
\label{eq:condimom:conti:n}
 m_k(x)=  a(m_k \ast F_{X_1})(x)+\sum_{j=0}^{k-1}\Bigg\{
 a \binom{k}{j} +b \binom{k-1}{j}
\Bigg\} (m_j\ast F_{X_1})(x),\quad x\ge 0,
\end{align}
where the operation $\ast$ denotes the convolution as usual.
\end{theorem} 

\begin{proof}
 Since $N$ belongs to the Panjer $(a,b)$ class, we can write
 \begin{align*}
  m_k(x) &= \sum_{n=0}^\infty n^k q_n F_{X_1}^{\ast (n)}(x) 
%         &= \sum_{n=0}^\infty (n+1)^k \Big(a+\frac{b}{n+1}\Big) q_n
%  F_X^{\ast (n+1)} \\
         = a \sum_{n=0}^\infty (n+1)^k q_n F_{X_1}^{\ast (n+1)}(x) +
  b\sum_{n=0}^\infty (n+1)^{k-1} q_n F_{X_1}^{\ast(n+1)}(x),
 \end{align*}
 where $F_{X_1}^{\ast n}(x)$ denotes the distribution of the $n$-th convolution of
 $X_1$.  
 Using the binomial expansion and changing the order of summations, we
 obtain
 \begin{align*}
  m_k(x) &= a \sum_{n=0}^\infty \sum_{j=0}^k \binom{k}{j}\, n^j q_n
  F_{X_1}^{\ast(n+1)}(x) + b \sum_{n=0}^\infty
  \sum_{j=0}^{k-1}\binom{k-1}{j}\, n^j q_n F_{X_1}^{\ast(n+1)}(x) \\
  &= a \sum_{n=0}^\infty n^k q_n F_{X_1}^{\ast(n+1)} + \sum_{n=0}^\infty 
 \Bigg\{
 a \sum_{j=1}^{k-1}\binom{k}{j}\, n^j q_n F_{X_1}^{\ast(n+1)}(x)+ b
  \sum_{j=0}^{k-1}\binom{k-1}{j}\, n^j q_n F_{X_1}^{\ast(n+1)}(x)
 \Bigg\} \\
  & = a (m_k\ast F_{X_1})(x) + \sum_{j=0}^{k-1} \Bigg\{
 a \binom{k}{j} + b\binom{k-1}{j} 
 \Bigg\} \sum_{n=0}^\infty n^j q_n F_{X_1}^{\ast(n+1)}(x). 
\end{align*}
\end{proof}
In view of expression \eqref{eq:condimom:conti:n}, the integral equation seems useless to obtain
$m_k(x)$ and we need additional techniques such as discretization of the 
density function of $X_1$ as in \cite[Example (p.123)]{RSST:1999}. 
However, it is helpful to obtain the generating function of $m_k$ by
providing an efficient recursion. 
%Here we take another
%path. From \eqref{eq:condimom:conti:n}, we can obtain ch.f of $m_k$ recursively and then we
%invert ch.f. to obtain $m_k$. 

\begin{lemma}\label{lem:chf:recursion}
 Assume that $N$ belongs to the Panjer $(a,b)$ class and iid r.v.'s $(X_i)$ take
 values on $\R_+$ with common ch.f. $\phi_{X_1}$. Then the FST of %Fourier transform of
 $m_k(x)=E[N^k; S_N\le x],\,k\in\N$ has the following form
 \begin{align*}
  \phi_{m_k}(u) &= \frac{1}{1-a\phi_{X_1}(u)} \sum_{j=0}^{k-1} \Bigg\{
 a \binom{k}{j}+b \binom{k-1}{j}
 \Bigg\} \phi_{m_j}(u)\,\phi_{X_1}(u),\quad u\in \R.
 \end{align*}
\end{lemma}
The proof of Lemma is a straightforward calculation and we omit it. 
Notice that due to Lemma \ref{lem:chf:recursion}, 
$\phi_{m_k}(u)$ can be presented by a combination of $G_N(\phi_{X_1}(u))$
and $\phi_{X_1}(u)$ 
since $\phi_{m_0}(u)=E[e^{iuS_N}]%P(S_N\le x)
=G_N(\phi_{X_1}(u))$. 
%Alternatively one can obtain
%\[
% \phi_{m_k}(u) = \int_{-\infty}^\infty e^{iux}\d m_k(x) =E[N^k
% \int_{-\infty}^\infty e^{iux}\d P(S_N\le x)]=E[N^k \phi_1^N(u)]
%\]
%where the iidness of $(X_i)$ and Fubini's theorem are applied, and then
%we apply the property of the Panjer class.
%For important conditional mean and variance, quantities $m^k(x)$ with $k=1,2$ are enough,
%which are 
%\begin{align*}
% m_1(x) &= a (m_1\ast F_X)(x) +(a+b)(m_0\ast F_X)(x), \\
% m_2(x) &= a (m_2 \ast F_X)(x) +(2a+b)(m_1 \ast F_X)(x) +(a+b)m_0\ast
% F_X(x) \\
%        &= a (m_2 \ast F_X)(x) + b (m_1\ast F_X)(x) +m_1(x),
% \end{align*}
%where $m_0(x)=P(S_N\le x)$. 
For a general $N$, we directly calculate the FST of $m_k$.

\begin{proposition}
 Let $N$ be a r.v. on $\N_0$ and let $(X_j)$ be an iid sequence of
 r.v.'s on $\R_+$. Assume $EN^k<\infty$ for $k\in\N$, then the FST of $m_k(x):=E[N^k;S_N\le x]$ has the form 
 \begin{align}
 \label{eq:Fourier:conti}
  \int_{0}^\infty e^{iux} \d m_k(x)=E[N^k\phi_{X_1}^N(u)]=\sum_{j=1}^k  \left\{ 
\begin{array}{c}
      k  \\
      j   
    \end{array}
\right\} \phi_{X_1}^j(u)\,\phi_{N}^{(j)}(\phi_{X_1}(u)),
 \end{align} 
where the left integral exists in the sense of 
the improper Riemann-Stieltjes integral.
% and 
%braces $\{\}$ denote the Striling number of the second kind and
% $\phi_Y(u)$ denotes ch.f. of r.v. $Y$. 
\end{proposition}

\begin{proof}
 Observe that $m_k(x)$ is a bounded non-decreasing function and $e^{iux}$ is
 continuous for every $u\in\R$, then a Riemann-Stieltjes integral
 $\int_0^M e^{iux}\d m_k(x)$ exists for all $M>0$ \cite[(2.24)
 Theorem]{wheeden:zygmund:1977}. Moreover, integration by parts (twice) and
 Fubini's theorem yield
 \begin{align*}
  \int_0^M e^{iux}\d m_k(x) &= [e^{iux}m_k(x)]_0^M -\int_0^M iu e^{iux}
  m_k(x)\d x \\
  &= [e^{iux}m_k(x)]_0^M -E\left[
 N^k \int_0^M iu e^{iux} P(S_N\le x \mid N)\d x 
 \right]\\
  &= [e^{iux}m_k(x)]_0^M -E\left[
 N^k \Big\{
 [e^{iux}P(S_N \le x \mid N)]_0^M -\int_0^M e^{iux}\d P(S_N \le x \mid N)
 \Big\}
\right]\\
 &= E\left[
 N^k \int_0^M e^{iux}\d P(S_N\le x \mid N)
 \right],
 \end{align*}
where in the third step, we use $m_k(x):=E[N^k P(S_N \le x \mid N)]$ for
 all $x\ge 0$. To obtain the first equality of \eqref{eq:Fourier:conti}
 take the limit $M\to\infty$ on both side, where in the right-hand side,
 the limit and expectation are exchangeable due to $EN^k<\infty$. Since
 the third equality follows similarly as in the proof of Proposition \ref{thm:fourier:inversion}, we
 conclude the result. 
\end{proof}
The inversion of FST $\mathcal F^{-1}$ is well known \cite[Theorem
4.4.1]{kawata:1972}: Let $\phi(u)$ be the FST of a bounded
non-decreasing function $F(x)$, then $\mathcal
F^{-1}$ is defined as 
\begin{align*}
 F(x)=\mathcal F^{-1}[\,\phi(\cdot)\,](x)=\lim_{T\to\infty}
 \frac{1}{2\pi}\int_{-T}^T [(e^{-ixt}-1)/-it]\phi(t)\d t,\quad x>0.
\end{align*}
Hence, $m_k(x)=\mathcal F^{-1}[\,\phi_{m_k}(\cdot)\,](x)$. 
%\begin{example}
 When a r.v. $N$ is a Poisson with parameter $\lambda$, so that
 $a=0,\,b=\lambda$ in Theorem \ref{thm:conti:panjer}, 
 we obtain
% \begin{align*}
$  m_1(x) = \lambda (m_0\ast F_{X_1})(x)$ and $ 
 m_2(x) = \lambda (m_1\ast F_{X_1})(x) + m_1(x)$.
% \end{align*}
% Those quantities are obtained by the inversion
% $\mathcal F^{-1}$ of the Fourier transforms,
Thus, it follows that
 \begin{align*}
  m_1(x) = \lambda \mathcal F^{-1}[\,\phi_{X_1}(\cdot)e^{\lambda (\phi_{X_1}(\cdot)-1)}\,](x), \qquad
  m_2(x) = \lambda^2 \mathcal F^{-1}[\,\phi_{X_1}^2(\cdot)e^{\lambda
  (\phi_{X_1}(\cdot)-1)}\,](x) + m_1(x). 
 \end{align*}
%where $\phi_X(u)$ is the ch.f. of a common distribution of $(X_j),\,j\in\N$.
%\end{example}

\subsection{Estimation of magnitude of each iid component}
A direct application of the Panjer recursion seems difficult for
both $\chi_k$ and $\chi_{k+}$, 
and we alternatively invert FST %Fourier transforms 
of these functions. 
%, the
%approach take in the discrete case (Sec. \ref{sec:discrete:element}). 
In order to obtain the FST, we represent $\chi_k$ and
$\chi_{k+}$ in the form of a convolution. 
 %as expressions by convolutions.
%we utilizes the property of convolution, which yields 
%The following proposition
%utilizes the idea of the convolution for general non-negative integer
%r.v. $N$. 

\begin{lemma}
\label{lem:continuous:x:sn}
 Let $N$ be a r.v. on $\N_0$ and $(X_i)$ is an iid
 sequence of r.v.'s on $\R_+$ with common ch.f. $\phi_{X_1}$ such that $EX_1^k  <\infty$. Then
 $\chi_k(x)=E[X_1^k;S_N\le x]$, $k\in \N$ has the form,
 \begin{align*}
  \chi_k(x) &= \frac{1}{i^k} \mathcal F^{-1} [\,\phi_{S_N}(\cdot)\,
  \phi_{X_1}^{(k)}(\cdot)/\phi_{X_1}(\cdot)\,](x). 
 \end{align*}
\end{lemma}

\begin{proof}
 %Due to Fubini's theorem assured by $EX_j^k<\infty$, we observe 
 We exploit the expression
 \begin{align*}
  \chi_k(x) = E[X_1^k;S_N \le x] %\\
%            &= E[X_1^k P(S_{N-1} \le x-X_1 \mid X_1,N)] \\
%            &= E \int_0^x y^k P(S_{N-1} \le x-y \mid N) \d P_{X_1}(y) \\
            = \int_0^x y^k P(S_{N-1} \le x-y) \d P_{X_1}(y), 
 \end{align*}
 %the last form of 
 which is the convolution of $P(S_{N-1} \le \cdot)$ and
 $\int_0^\cdot y^k \d P_{X_1}(y)$. Since the ch.f. of $S_{N-1}$ is
 $\phi_{S_N}(u)/\phi_{X}(u)$, the conclusion is implied by the
 FST of $\int_0^\cdot y^k \d P_{X_1}(x)$ which is 
$ \int_0^\infty e^{iux} x^k \d P_{X_1}(x)=i^{-k}\phi_{X_1}^{(k)}(u) $, 
%\begin{align*}
%\int_0^\infty e^{iux} x^k \d P_{X_1}(x) = \frac{1}{i^k}
% \phi_{X_1}^{(k)}(u),
%\end{align*} 
where $EX_1^k<\infty$ assures the existence of $\phi_{X_1}^{(k)}(u)$. 
\end{proof}

 The corresponding result for $E[X_1^k;S_{N+1}\le x]$ is obvious. Under
 the same condition of Lemma \ref{lem:continuous:x:sn}, we have 
% \begin{align*}
$  \chi_{k+}(x) = i^{-k}%\frac{1}{i^k} 
\mathcal F^{-1}[\,\phi_{S_N}(\cdot)\,
  \phi_{X_1}^{(k)}(\cdot)\,](x)$. 
% \end{align*}

\section{Numerical Examples}
\label{sec:numeraical:ex}
%For convenience, we prepare notations of distributions used in
%examples. 
%Denote a Poisson distribution with parameter $\lambda$ by
%$Pois(\lambda)$ and by $Geo(p)$, a geometric distribution with
%paprameter $p$ of which probability function is
%$P(X>k)=pq^k,\,q=1-p,\,k\in \N_0$ for $X\sim Geo(p)$. 
 
We prepare notations of distributions used in examples. Denote a Poisson
distribution with parameter $\lambda$ by $Pois(\lambda)$ and by
$Geo(p)$, a geometric distribution with parameter $p$ of which
probability is $P(X=k)=pq^k,\,q=1-p,\,k\in \N_0$. As usual write
$X\sim\cdot$ if r.v. $X$ follows the distribution after the tilde. All
computations are done with Mathematica ver. 9 of Wolfram. 

Firstly a simple example of $E[N\mid S_N]$ is presented by setting $N\sim
Pois(\lambda)$ and $X_1%$ of $S_N=\sum_{j=1}^N X_j$ $
\sim Pois(\gamma)$. We examine two proposed methods for $m_1=E[N;S_N]$, the recursion
method and the Fourier inversion. 
% where calculated results coinsiders (the difference
%negligibly small). 
For the probability of $S_N$, we use the ordinary recursion (Theorem
\ref{thm:basic}), which yields
\begin{align}
\label{eq:recursion:ex1}
 P(S_N=\ell) = \left \{
\begin{array}{ll}
E[P(X_1=0)^N] = e^{\lambda(e^{-\gamma}-1)}, &\ell=0, \\
\sum_{j=1}^\ell \frac{\lambda j}{\ell} P(X_1=j)P(S_N=\ell-j), & \ell\ge 1. 
\end{array}
\right. 
\end{align}
%Then for $m_1(\ell)=E[N;S_N=\ell]$, 
We apply Theorem \ref{thm:reccursion-n} to obtain the recursion,  
\begin{align}
\label{eq:recursion:ex2}
 m_1(\ell) = \left \{
\begin{array}{ll}
E[NP(X_1=0)^N] = \lambda e^{-\gamma} e^{\lambda(e^{-\gamma}-1)}, &\ell=0, \\
\sum_{j=1}^\ell \frac{\lambda j}{\ell}
 P(X_1=j)\{P(S_N=\ell-j)+m_1(\ell-j)\}, & \ell \ge 1. 
\end{array}
\right. 
\end{align}
Another method for $m_1(\ell)$ is 
to apply \eqref{eq:Fourier:inversion} to
the Fourier transform (Proposition
\ref{thm:fourier:inversion}), which is
%\begin{align}\label{eq:F-transform}
% m_1(\ell) = \frac{1}{2\pi} \int_{-\pi}^\pi e^{-i\ell u} G_{m_1}(e^{iu})du,
%\end{align}
%where 
\begin{align*}
 G_{m_1}(e^{iu}) = ENG_{X_1}^N(e^{iu})= ENe^{\gamma(e^{iu}-1)N} = \lambda
 e^{\gamma(e^{iu}-1)}e^{\lambda(e^{\gamma(e^{iu}-1)}-1)}. 
\end{align*}
In Figure \ref{fig:1}, we plot
$E[N\mid S_N=\ell]=m_1(\ell)/P(S_N=\ell),\,\ell \ge 0$ using both
methods.  
%recursion 
%\eqref{eq:recursion:ex2} and %the Fourier inversion 
%\eqref{eq:F-transform}. 
\begin{figure}
\begin{center}
\includegraphics[width=0.44\textwidth]{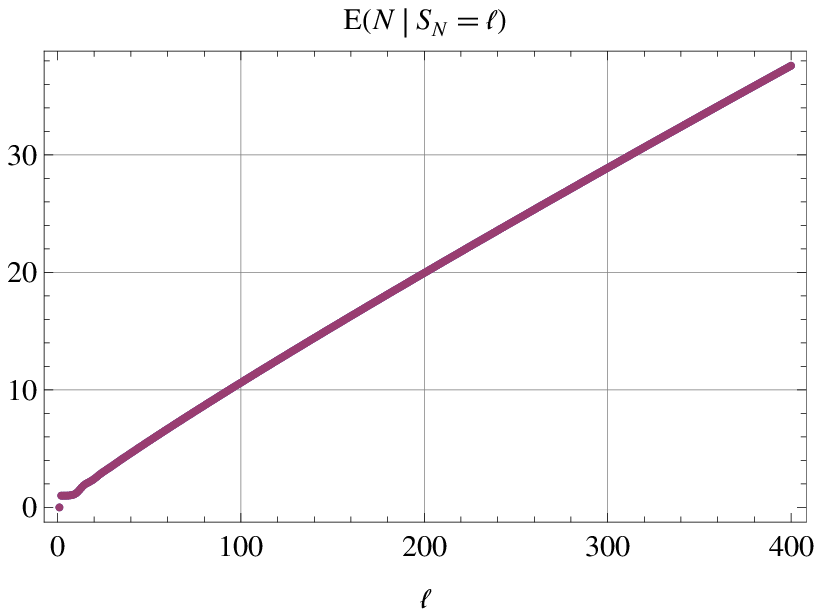} \hspace{4mm}
\includegraphics[width=0.45\textwidth]{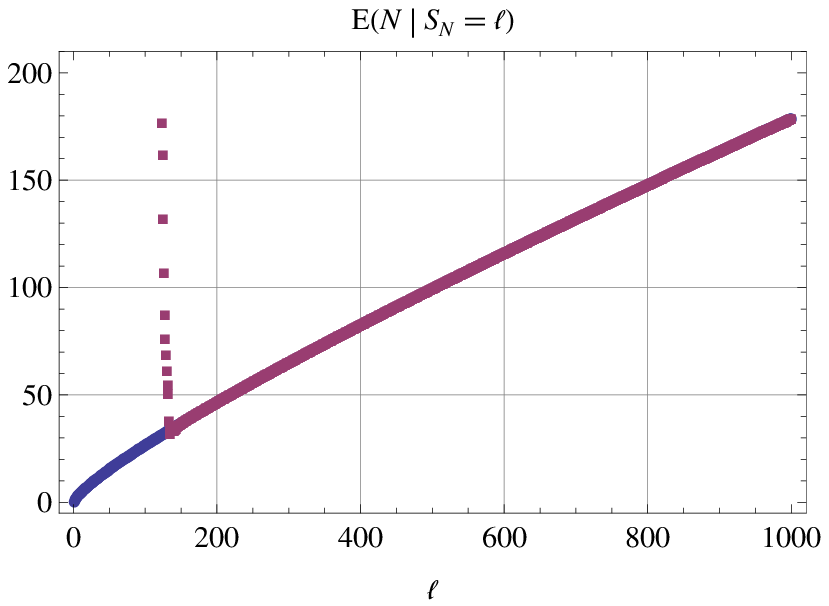}
\end{center}
\caption{Left: 
the conditional moment $E[N \mid S_N =\ell],\,\ell\in [0,400]$ when 
Poisson
 parameters for $(N,X_1)$ are ($\lambda=20,\,\gamma=10$). Right: $E[N \mid S_N =\ell],\,\ell\in [0,1000]$ under the setting
 ($\lambda=100,\,\gamma=5$). Squared dots are values by the Fourier
 approach and round dots are those by the recursion method.
In the former case
 values of both methods coincide. However, in the latter case
 instability is observed in the Fourier method for small $\ell$, though
 for large $\ell$ they coincide.}
\label{fig:1}
\end{figure}
Although they coincide when parameters are moderate, if
either of parameters of Poisson for $N$ and $X_1$ is large, 
we observe instability 
%of $E[N\mid S_N=\ell]$ 
for small $\ell$ in the Fourier approach (Figure \ref{fig:1}: Right,
squared dots), though for large $\ell$ there is no difference. 
 
%In Figure \ref{fig:1}, 
Next, we consider an example of $E[X_1\mid S_{N+1}]$ for the recursion
and that of $E[X_1\mid S_N]$ with the Fourier transform, where $N \sim
Pois(\lambda)$ and 
$X_1\sim Geo(p)$. %follows a geometric distribution
      %$P(X_1=k)=pq^k,\,q=1-p,k\in\N_0$. 
The Panjer recursion is 
applied to both $P(S_N=\cdot)$ and $P(S_{N+1}=\cdot)$, for the latter of
which we also use the convolution. 
For $\chi_{1+}(\ell)=E[X_1;S_{N+1}=\ell]$, we use the recursion by Theorem
\ref{thm:element:fourier}, i.e. 
\begin{align*}
 \chi_{1+}(\ell) = \left \{
\begin{array}{ll}
p q e^{-\lambda q}, &\ell=1, \\
\ell p q^\ell e^{-\lambda q} +
\sum_{j=1}^{\ell-1} 
p q^j \frac{\lambda j}{\ell-j} \chi_{1+}(\ell-j), & \ell\ge 2. 
\end{array}
\right. 
\end{align*}
For $\chi_1(\ell)=E[X_1;S_N=\ell]$, the inversion of the Fourier
transform \eqref{eq:Fourier:inversion} is applied to 
\[
 G_{m_1}(e^{iu})=\frac{qe^{iu}}{1-qe^{iu}} \exp\Big\{
 \lambda\frac{q(1-e^{iu})}{qe^{iu}-1}
\Big\}.
\]
\begin{figure}
\begin{center}
\includegraphics[width=0.45\textwidth]{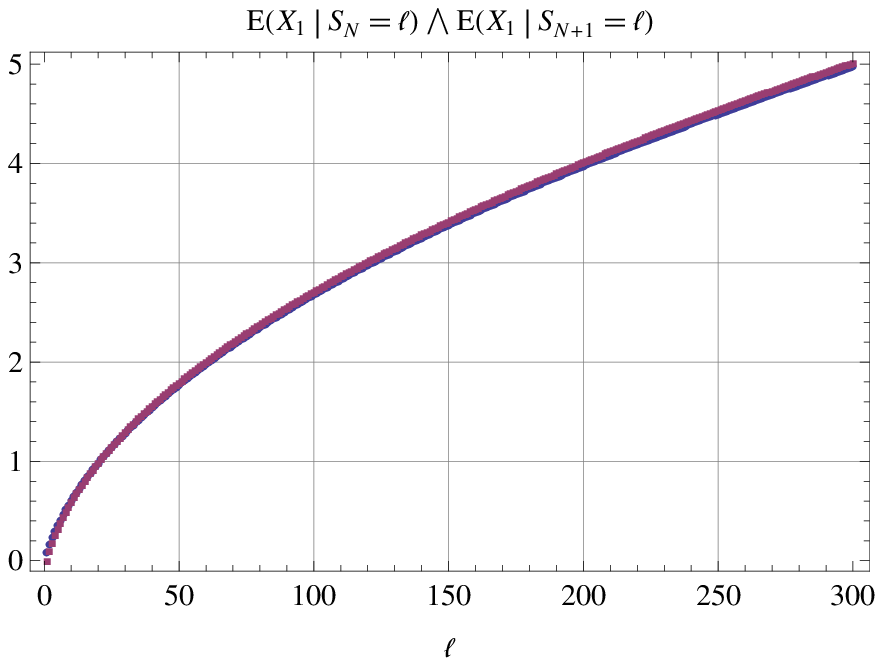}\hspace{4mm}
\includegraphics[width=0.45\textwidth]{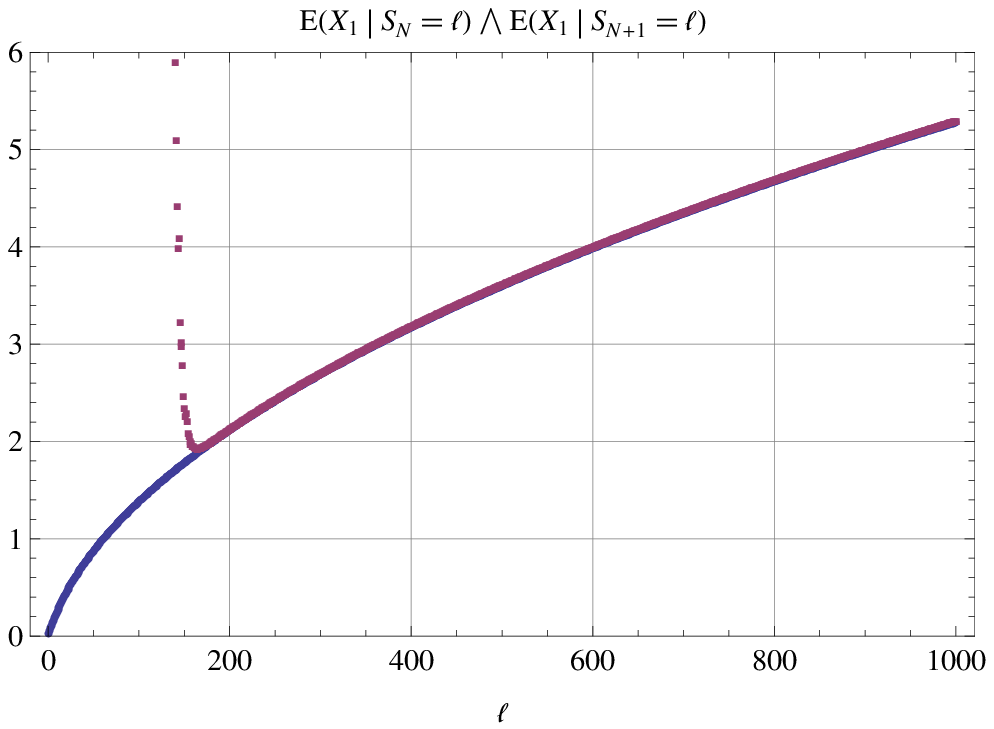}
\end{center}
\caption{Left: 
the conditional moments $E[X_1 \mid S_{N}=\ell]$ (square dots) and $E[X_1 \mid
 S_{N+1}=\ell]$ (round dots) for $\ell\in [1,300]$ 
when parameters of $Pois(\lambda)$ and $Geo(p)$ for $(N,X_1)$ are
 ($\lambda=40,\,p=0.25$) respectively. Right: the same quantities of the
 left but with $\ell\in [1,1000]$ and ($\lambda=150,\,p=0.2$). 
 In both graphs these quantities present quite similar curves for large
 $\ell$. However, in the right graph instability occurs in small $\ell$
 of $E[X_1 \mid
 S_{N}=\ell]$ (the Fourier approach).}
\label{fig:2}
\end{figure}
In Figure \ref{fig:2}, we plot $E[X_1\mid S_N=\ell]$ and $E[X_1\mid
S_{N+1}=\ell]$ for $\ell \ge 1$. Since the graphs show very similar
curves for a moderate setting
of parameters, we conclude that both methods work properly. However the
instability is again observed in the Fourier approach (Figure
\ref{fig:2}: Right, squared dots) %, i.e. calculation
%of $E[X_1\mid S_N=\ell]$ 
when the parameter $\lambda$ is large and $\ell$
is small.  

\subsection{Prediction in Poisson shot noise process}
We pursue the prediction $E[M(t,t+s]\mid M(t)],\,t,s>0$ of the model \eqref{def:posnp},
i.e. calculate the quantity $E[N(t)\mid M(t)]$ in \eqref{predi:posnp}. As mentioned,
since the order of $(T_j)$ in \eqref{def:posnp} does not change the distributional
relation of $N(t)$ and $M(t)$, by the order statistics property of the
Poisson, we may consider the model $M(t):=\sum_{k=1}^{N(t)}L_k(t-U_k)$
with the iid $U(0,t)$ sequence $(U_i)$, and then study $E[N(t)\mid
M(t)]$. We assume that the processes $L_k$'s are iid compound Poisson
processes such that the generic process $L$ has the form
$L(t)=\sum_{j=1}^{N_0(t)}Y_j$, where $N_0(t)\sim
Pois(\gamma t)$, and $(Y_j)$ denotes an iid sequence of non-negative
jump sizes. 

%We pursue the example of Poisson cluster models where 
%the model $M(t)=\sum_{k=1}^{N(1)}L_k(t-U_k)$ corresponds to the random sum
%$S_N$ with $N:=N(1)$ and $X_k=L_k(t-U_k)$. Here an iid sequence of L\'evy processes
%$(L_k)$ and an iid $U(0,1)$ sequences $(U_k)$ are independent. 
%Recalling that the compound Poisson (CP for short) process is fundumental process among
%L\'evy processes, 
%we assume that processes $L_k$'s are iid CP prcesses 
%such that the generic process is given by $L(t)=\sum_{j=1}^{N_0(t)}Y_j$,
%where the underlying Poisson process $N_0(t)$ follows $Pois(\gamma t)$ and $(Y_j)$
%denotes iid non-negative jump sizes. 
%where the
%underlying Poisson process $N_0(t)$ follows $Pois(\gamma t)$ and $(Y_j)$
%denotes iid non-negative jump sizes so that each cluster has the form $ L_1(t-U)=\sum_{j=0}^{N_0(t-U_1)} Y_j $. 
%Note that the compound Poisson processes are fundumental process among
%L\'evy processes. 
%Firstly we try
%the reccurision to obtain the probablity of 
Now by setting $N:=N(t)$ and
$X_i:=L_i(t-U_i),\,i\in \N$%=\sum_{j=1}^{N_0(t-U_1)}Y_j$
, the calculation of
$E[N(t)\mid M(t)]$ can be considered in the framework of $E[N\mid
S_N]$. For the probability of $X_1$, since $N_0(t-U_1)$ does not belong
to the Panjer class, we take the Fourier approach. For this we need
ch.f. of $X_1:=L_1 (t-U_1)$, which is 
\[
 E[e^{iuX_1}] = E[e^{iuL(t-U_1)}] =E \big[ e^{iu \sum_{j=1}^{N_0(t-U_1)} Y_j} \big] = \frac{e^{ \gamma
 (t-1) (\phi_{Y_1}(u)-1)}-e^{\gamma t(\phi_{Y_1}(u)-1)}}{\gamma (1-\phi_{Y_1}(u))},
\]
where $\phi_{Y_1}(u)$ is the ch.f. of $Y_1$. Thus after putting $Y_1\sim
Pois(\mu)$ so that $\phi_{Y_1}(u)=e^{\mu(e^{iu}-1)}$ we obtain the
probability of $X_1$ by the Fourier inversion. For simplicity, we set
$t=1$, i.e. consider $E[N(1)\mid M(1)]$, and apply Theorem
\ref{thm:reccursion-n} or equivalently apply the recursions \eqref{eq:recursion:ex1}
and \eqref{eq:recursion:ex2} with initial values $P(S_N=0)=e^{\lambda(P(X_1=0)-1)}$ and
$m_1(0)=E[NP(X_1=0)^N]=\lambda P(X_1=0)P(S_N=0)$.

%Although we try a reccursion method for the probability of 
%$X_k \stackrel{d}{=} L_1(t-U) = \sum_{i=1}^{N_0(t-U)}Y_j$ with $U\sim U(0,1)$, 
%which is a variant of a CP distribution,
%some numeraical unstablility occurs, 
%and we take the Fourier inversion approach. 
%The ch.f. of $X_k\stackrel{d}{=}L(t-U)$ is 
%\[
% E[e^{iuX_k}] = E \big[ e^{iu \sum_{j=1}^{N_0(t-U)} Y_j} \big] = \frac{e^{ \lambda
% (t-1) (\phi_{Y_1}(u)-1)}-e^{\lambda t(\phi_{Y_1}(u)-1)}}{\lambda(1-\phi_{Y_1}(u))},
%\]
%where $\phi_{Y_1}(u)$ is the ch.f. of $Y_1$. Now after putting $Y_1\sim Pois(\mu)$
%so that $\phi_{Y_1}(u)=e^{\mu(e^{iu}-1)}$, we obtain the probability of
%$X_k$ by the Fourier inversion. 
%and take $t=1$, we obtain 
%\[
% E \big[
% e^{iu \sum_{j=1}^{N_0(t-U)}Y_j}
%\big] = \frac{1-e^{\gamma(\phi_{Y_1(u)}-1)}}{\gamma(1-\phi_{Y_1}(u))}
%=\frac{1-e^{\gamma(e^{\mu(e^{it}-1)}-1)}}{\gamma(1-e^{\mu(e^{it}-1)})}
%\]
\begin{figure}
\begin{center}
\includegraphics[width=0.45\textwidth]{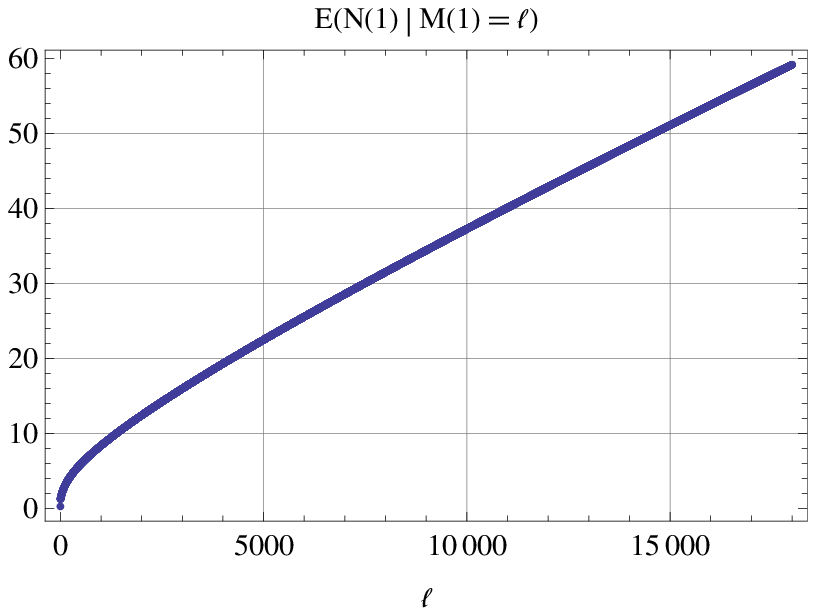}
\includegraphics[width=0.45\textwidth]{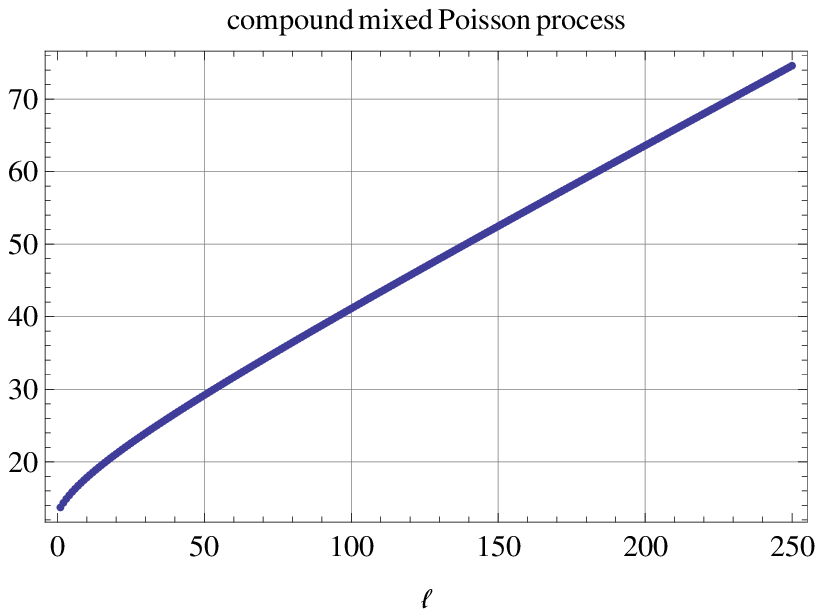}
\end{center}
\caption{Left: we plots $E[N(1)\mid M(1)=\ell],\,\ell\in[0,18000]$ which
 is the quantity in the predictor of a Poisson shot noise process of \eqref{predi:posnp}
 with $t=1$. Right: we plot
 the predictor of a compound mixed Poisson process, $E[Z(1,2]\mid
 Z(1)=\ell],\,\ell\in [0,250]$. In both graphs one see the non-linear
 curves which show that the linear estimations are insufficient.}
\label{fig:3}
\end{figure}
%In table \ref{fig:3}, we plots the probablity of $L_1(t-U)$ under the setting. Now
%we set $N(1)\sim Pois(\lambda)$ and apply Theorem 2.1 or equivalently 
%recursions \eqref{eq:recursion:ex1} and \eqref{eq:recursion:ex2} with 
%$P(S_N=0)=E[P(X_1=0)^N]=e^{\lambda(P(X_1=0)-1)}$ and
%$m_1(0)=E[NP(X_1=0)^N]=\lambda P(X_1=0)P(S_N=0)$ to obtain the quantity
%$E[N(1)\mid M(t)]$ with $t=1$. 
In Figure \ref{fig:3} (left), we plots $E[N(1)\mid M(1)=\ell]$ for $\ell\in[0,18000]$
with $\gamma=100,\,\mu=5$ and $\lambda=30$. In view of the graph, our
computational method seems to work well, and one can see a non-linear
curve which shows that the linear estimation of $N(t)$ by $M(t)$ is insufficient.

\subsection{Prediction in compound mixed Poisson process}
We consider an example of the compound mixed Poisson process mixed by a Gamma
r.v. called compound P\'olya process \cite[Ex. 4.1]{grandell:1997}. 
Let $\overline N(t):=\pi(\theta \Lambda(t))$ denotes a mixed Poisson
process where $\pi(t)$ be a homogeneous Poisson process with intensity
$1$ on $[0,\infty)$, $\Lambda(t)$ is an intensity measure and $\theta$
is a Gamma $(\alpha,\beta)$ r.v. of which density is
$f_\theta(x)=\frac{\beta^\alpha}{\Gamma(\alpha)}x^{\alpha-1}e^{-\beta
x}$. Then 
the process has the form
$Z(t)=\sum_{j=1}^{\overline N(t)} X_j,\,t>0$, where $X_j$'s are iid r.v.'s on
$\N_0$ or $\R_+$ such that $\overline N$ and $(X_j)$ are
independent. 
Since the $\sigma$-fields $\mathcal{G}_t$ by $\{\overline N(t),\,\overline N(t,t+s],\,Z(t)\},\,t,s>0$ and
$\mathcal{H}_t$ by $\{\overline N(t),\,Z(t)\}$ are finer than that by
$\{Z(t)\}$, the conditional expectation of increments
$Z(t,t+s]:=Z(t+s)-Z(t)$ given $Z(t)$ has  
\begin{align}
 E[Z(t,t+s]\mid Z(t)] %&=E[X_1]E[\overline{N}(t,t+s]\mid Z(t)] \nonumber \\
 &=E[E[E[Z(t,t+s]\mid \mathcal{G}_t] \mid \mathcal{H}_t]\mid Z(t)] \nonumber \\ 
 &=E[X_1] E[E[\overline N(t,t+s]\mid \mathcal{H}_t ]\mid Z(t)] \label{eq:mxpc}\\
 &=E[X_1] E[E[\overline N(t,t+s]\mid \overline N(t)]\mid Z(t)],\nonumber
\end{align}
where in the third step we use the conditional independence of $\overline
N(t,t+s]$ and $Z(t)$ given
$\overline N(t)$ (\cite[Prop. 6.6]{kallenberg:2002}). 
%Since 
%\begin{align*}
% E[\theta^m e^{-\theta \Lambda(t)}] = \frac{\beta^\alpha}{(\Lambda(t)+\beta)^{\alpha+m}}\frac{\Gamma(m+\alpha)}{\Gamma(\alpha)},
%\end{align*}
%we have 
Since 
\begin{align*}
 E[\overline{N}(t,t+s]\mid \overline{N}(t)=m] &= \sum_{k=0}^\infty k
 \frac{\Lambda^k(t,t+s]}{k!} \frac{E[\theta^{k+m}e^{-\theta
 \Lambda(t+s)}]}{E[\theta^m e^{-\theta \Lambda(t)}]} \\
&= \Lambda(t,t+s] \frac{E[\theta^{m+1}e^{-\theta
 \Lambda(t)}]}{E[\theta^m e^{-\theta \Lambda(t)}]} \\
&=\Lambda(t,t+s] \frac{\alpha+m}{\Lambda(t)+\beta},
\end{align*}
where in the second step we exchange the infinite sum and the expectation
operator (see also \cite[(1.4)]{grandell:1997}), we proceed the calculation \eqref{eq:mxpc} to get 
% Substitute this into \eqref{eq:mxpc} in order to get 
\[
 E[Z(t,t+s]\mid Z(t)] =
 \frac{E[X_1]\Lambda(t,t+s]}{\Lambda(t)+\beta}(\alpha+E[\overline{N}(t)\mid
 Z(t)]). 
\]
Now let $\Lambda(t):=t,\,\beta:=1,\,\alpha:=7$ and $X_1\sim$ Geo $(1/4)$,
we obtain 
%follows a
%geometric distribution $P(X_1=k)=pq^k,\,q=1-p,k\in\N_0$, it follows that 
\[
 E[Z(t,t+s]\mid Z(t)] = %\frac{q}{p}
\frac{3 s}{1+t}(7+E[\overline{N}(t)\mid Z(t)]).
\]
Since $\overline N(t)$ does not belong to the Panjer class, we apply the
Fourier approach. 
Due to Proposition \ref{thm:fourier:inversion} together with
\[
 G_{S_N}(u)=E[G_{X_1}^{\overline N (t)}(u)] = \frac{1}{1+t(1-G_{X_1}(u))}
\quad \mathrm{and} \quad G_{m_1}(u)=
 E[\overline N G_{X_1}^{\overline{N}}(u)]= \frac{7 G_{X_1}(u)}{\{1+t(1-G_{X_1}(u))\}^8},
\]
we obtain the quantity $E[\overline{N}(t)\mid Z(t)]$ by the inversion
formula \eqref{eq:Fourier:inversion}. In Figure \ref{fig:3} (right), we plot
$E[Z(t,t+s]\mid Z(t)=\ell],\,\ell\in [0,250]$ with $s=t=1$, where one would again
observe a non-linear curve.  

\appendix
\section{Calculation of \eqref{predi:posnp}}\label{append:sec}
For the calculation of \eqref{predi:posnp}, we use the following properties.\\
(a). By definition, the $\sigma$-field by $M(t)$ is included in the
$\sigma$-filed by $(L_k(t-T_k))$ and $(T_k)$. \\
(b). Since $N(t)=\sum_{k=1}^\infty I_{(T_k\le t)}$, the $\sigma$-field 
by $N(t)$ is included in the $\sigma$-field by $(T_k)$. \\
(c). By the order statistics property of a Poisson, given $N(t)$ and
$N(t+s)$, the set of points $(T_k)\in(0,t]$ and the set of points
$(T_k)\in(t,t+s]$ are independent. \\
(d). Given $N(t,t+s]$, points $(T_k)\in (t,t+s]$ are mutually
independent. \\
(e). Stationary and independent increments of L\'evy processes. \\
By a multiple use of iterated property of the conditional expectation \cite[Theorem 6.1 (vii)]{kallenberg:2002},
detailed the calculation of \eqref{predi:posnp} is 
\begin{align*}
 & E[M(t,t+s]\mid M(t)] \\
 &= E[\sum_{j=N(t)+1}^{N(t+s)}L_j(t-T_j, t+s-T_j]\mid M(t)]
 + E[\sum_{j=1}^{N(t)}L_j(t-T_j,t+s-T_j]\mid M(t)] \\
 & = E[\sum_{j=N(t)+1}^{N(t+s)}E[L_j(t-T_j, t+s-T_j] \mid
 N(t),N(t+s),\{(T_k),(L_k(t-T_k))\}_{k:T_k\le t} ] \mid M(t)] \\
 & \quad + E[ \sum_{j= 1}^{N(t)}E [L_j(t-T_j, t+s-T_j]\mid
 \{(T_k),(L_k(t-T_k))\}_{k:T_k\le t}]\mid
 M(t)] \\
 & = E[\sum_{j=N(t)+1}^{N(t+s)}E[L_j(t+s-T_j) I_{(t< T_j \le t+s)}
 \mid N(t),N(t+s)]\mid M(t)] \\
 & \quad + E[\sum_{j=1}^{N(t)}E[L_j(t-T_j,t+s-T_j]\mid T_j,L_j(t-T_j)]\mid
 M(t)] \\
 & = E[\sum_{j=N(t)+1}^{N(t+s)} E[L(t+s-U)]\mid M(t)] +
 E[\sum_{j=1}^{N(t)} E[L(s)] \mid M(t)]\\
 & = E[N(t,t+s]\mid M(t)]E[L(t+s-U)] + E[L(s)]E[N(t)\mid M(t)], 
% & = E[N(s)] \mu(t+s-E[V])+ \mu s E[N(t)\mid M(t)],
\end{align*}
where in the second step, the properties (a) and (b) are used, and in the third step, we
exploit (c) and (e) so that the conditional independence of
$(L_j(t-T_j,t+s-T_j])_{j:t<T_j\le t+s}$ and $\{(T_i),(L_i(t-T_i))\}_{i:T_i\le
t}$. In the fourth step we use (d) and (e). 
%together with the order
%statistics property and (d). 
Finally since the quantity $N(t,t+s]$ is
independent of the $\sigma$-field $\mathcal F_t$ constructed by all available set
before $t$, the conclusion holds. \\

\noindent {\bf Acknowledgment:} 
The author would like to thank Prof. Tomasz Rolski for fruitful
discussion of the topic and hospitality when he visited the mathematical
institute in the University
of Wroclaw.


\begin{thebibliography}{99}
\baselineskip12pt

\bibitem{abramowitz:stegun:1972}
{\sc Abramowitz, M. and Stegun, I. A.}\ (1972) 
{\em Handbook of Mathematical Functions with Formulas, Graphs, and
	Mathematical Tables}. 
National Bureau of Standards Applied Mathematics Series {\bf 55}. Tenth Printing.

\bibitem{bartlett:1963}
{\sc Bartlett, M.S.}\ (1963)
The spectral analysis of \pp es.
{\em J. R. Statist. Soc. Ser. B. Stat. Methodol.} {\bf 25}, 264--296.

%\bibitem{carrilo:1989}
%{\sc Carrillo, M.J.}\ (1989)
%Generalizations of Palm's Theorem and Dyna-METRIC's Demand and Pipeline
%Variability. 
%{\em  Project AIR FORCE report
%prepared for the United States Air
%Force} {\bf R-3698-AF}.

%\bibitem{daley:vere-jones:1988}
%{\sc Daley, D.J. and Vere-Jones, D.} (1988)
%{\em An Introduction to the Theory of Point Processes}. 
%Springer, New York. 

%\bibitem{daley:vere-jones:2003}
%{\sc Daley, D. and Vere-Jones, D.} (2003)
%{\em An Introduction to the Theory of Point Processes Elementary Theory
%	and Methods}. Vol.{\bf 1} 
%Springer, Heidelberg.

%\bibitem{daley:vere-jones:2003}
%{\sc Daley, D.J. and Vere-Jones, D.} (2003)
%{\em An Introduction to the Theory of Point Processes. Elementary Theory
%	and Methods}. Vol.{\bf 1}. 2nd ed.
%Springer, New York. 

%\bibitem{daley:vere-jones:2008}
%{\sc Daley, D.J. and Vere-Jones, D.} (2008)
%{\em An Introduction to the Theory of Point Processes. General Theory and
%	Structure}. Vol.{\bf 2}. 2nd ed.
%Springer, New York. 

%FFFFFFFFFFFFFFFFFFFFFFFFFFFFFFFFFFFFFFFFFFFFFFFFFFFFFFFFFFF
%\bibitem{Gneiting:Schlather:2004}
%{\sc Gnaiting, T. and Schlather, M.} (2004)
%Stochastic models that separate fractal dimension and the Hurst effect. 
%{\em SIAM Rev.} {\bf 46}, 296--282. 

%\bibitem{hohn:veitch:abrey:2003}
%{\sc Hohn, N., Veitch, D. and Abry, P.} (2003)
%Cluster processes: a natural language for network traffic. 
%{\em IEEE Trans. Signal Process.} {\bf 51}, 2229--2244. 

\bibitem{embrechts:frei:2009}
{\sc Embrechts, P. and Frei, M.}\ (2009)
Panjer recursion versus FFT for compound distributions. 
{\em Math. Methods Oper. Res.} {\bf 69}, 497--508.

\bibitem{embrechts:kluppelberg:mikosch:1997}
{\sc Embrechts, P., Kl\"uppelberg, C. and Mikosch, T.}\ (1997)
{\em Modelling Extremal Events for Insurance and Finance.}
Springer, Berlin.

%\bibitem{fay:gonzalez:mikosch:samorodnitsky:2006}
%{\sc Fa\"y, G., Gonz\'alez-Ar\'evalo, B., Mikosch, T. and 
% Samorodnitsky, G.} (2006)
%Modeling teletraffic arrivals by a Poisson cluster process. 
%{\em Queueing Syst.} {\bf 54}, 121--140. 

\bibitem{feller:1968}
{\sc Feller, W.}\ (1968)
{\em An Introduction to Probability Theory and its Applications.} Vol. I,
3rd ed. Wiley, New York.

\bibitem{goldie:kluppelberg:1998}
{\sc Goldie, C. and Kl\:uppelberg} (1998)
{\em Subexponential distributions.} In: Adler, R. Feldman, R. and Taqqu,
	M. S. (Eds.)
{\em A Practical Guide to Heavy Tails: Statistical Techniques and
	Applications}, 435--460. Birkh\"auser, Boston.  

\bibitem{grandell:1997}
{\sc Grandell, J.}\ (1997)
{\em Mixed Poisson Processes.} Chapman and Hall, London.

\bibitem{gut:2009}
{\sc Gut, A.}\ (2009)
{\em Stopped Random Walks : Limit Theorems and Applications.}
2rd ed. Springer, New York.

\bibitem{jessen:mikosch:2006}
{\sc Jessen, A.H.,  Mikosch, T.}\ (2006)
Regulary varying functions. 
{\em Publications de L'institut Math\'ematique} {\bf 80}, 171-192. 

\bibitem{jessen:mikosch:samorodnitsky:2009}
{\sc Jessen, A.H.,  Mikosch, T. and Samorodnitsky, G.}\ (2011)
Prediction of outstanding  payments in a Poisson cluster model.
{\em Scand. Actuar. J.} {\bf 2011}, 214-237. 

%\bibitem{johnson:kotz:kemp:1992}
%{\sc Johnson, N., Kotz, S. and Kemp, A.}\ (1992)
%{\em Univariate Discrete Distributions}. Wiley Series in Probability and
%	Mathematical Statistics, second ed. Wiley, New York.

\bibitem{kallenberg:2002}
{\sc Kallenberg, O.}\ (2002)
{\em Foundations of Modern Probability.} 2nd ed. Springer, New York.

\bibitem{kawata:1972}
{\sc Kawata, T.}\ (1972)
{\it Fourier Analysis in Probability Theory.} Academic Press, New York.


%\bibitem{kingman:1993}
%{\sc Kingman, J.F.C}\ (1993)
%{\em Poisson Processes}, Claredon Press, Oxford.


%\bibitem{kozubowski:podgorski:2009}
%{\sc Kozubowski, T.J. Podg\'orski, K.} (2009)
%Distributional properties of the negative binomial L\'evy process. 
%{\em Probab. Math. Statist.}, {\bf 29}, 43-71. 

%\bibitem{konstantopoulos:lin:1998}
%{\sc Konstantopoulos, T. and Lin, Si-Jian}\ (1998)
%Macroscopic models for long-range dependent network traffic. 
%{\em Queueing Syst.} {\bf 28}, 215--243. 

%KKKKKKKKKKKKKKKKKKKKKKKKKKKKKKKKKKKKKKKK

\bibitem{kluppelberg:mikosch:1995a}
{\sc Kl\"uppelberg, C. and Mikosch, T.}\ (1995)
Explosive Poisson shot noise processes with applications to risk reserves.
{\em Bernoulli} {\bf 1}, 125--147. 

\bibitem{kluppelberg:mikosch:1995b}
{\sc Kl\"uppelberg, C. and Mikosch, T.}\ (1995)
Delay in claim settlement and ruin probability approximations.
{\em Scand. Actuar. J.} {\bf 1995}, 154--168.

%\bibitem{kluppelberg:mikosch:scharf:2003}
%{\sc Kl\"uppelberg, C., Mikosch, T. and  Sch\"arf, A.} (2003)
%Regular variation in the mean and stable limits for Poisson shot
%noise. 
%{\em Bernoulli} {\bf 9}, 467--496. 

%LLLLLLLLLLLLLLLLLLLLLLLLLLLLLLLLLLL
\bibitem{lewis:1964}
{\sc Lewis, P.A.}\ (1964)
A branching Poisson process model for the analysis of
computer failure patterns. 
{\em J. R. Stat. Soc. Ser. B Stat. Methodol.} {\bf 26}, 398--456.

%\bibitem{levy:taqqu:2000}
%{\sc Levy, J.B. and Taqqu, M.S.}\ (2000)
%Renewal reward processes with heavy-tailed interrenewal
%times and heavy-tailed rewards. {\em Bernoulli} {\bf 6}, 23--44.

%\bibitem{lundberg:1903}
%{\sc Lundberg, F.}\ (1903)
%{\it Approximerad framst\"allning av sannolikhetsfunktionen. 
%\r{A}terf\"ors\"akring av kollektivrisker.}
%Akad. Afhandling. Almqvist och Wiksell, Uppsara.

\bibitem{matsui:mikosch:2010}
{\sc Matsui, M. and Mikosch, T.} (2010)
Prediction in a Poisson cluster model. {\em J. Appl. Probab.} {\bf 47}, 350--366. 

\bibitem{matsui:2011}
{\sc Matsui, M.} (2011)
Prediction in a Poisson cluster model with multiple cluster
	processes. {\em Scand. Actuar. J.} {\bf 2015}, 1--31.  

\bibitem{matsui:2014}
{\sc Matsui, M.} (2014)
Prediction in a non-homogeneous Poisson cluster model.
 {\em Insurance Math. Econom.}, {\bf 55}, 10--17.

\bibitem{matsui:rolski:2014}
{\sc Matsui, M. and Rolski, T.} (2014)
Prediction in a mixed Poisson cluster model, preprint.

\bibitem{mikosch:2009}
{\sc Mikosch, T.}\ (2009)
{\em Non-Life Insurance Mathematics. An Introduction with the 
Poisson Process.} 2nd ed. Springer, Heidelberg.



%\bibitem{mikosch:samorodnitsky:2007}
%{\sc Mikosch, T. and Samorodnitsky, G.} (2007) 
%Scaling limits for cumulative input processes.  {\em Math. Oper. Res.}
%{\bf 32}, 890--919. 

%NNNNNNNNNNNNNNNNNNNNNNNNNNNNNNNNNNNNNNNNNNNNNNNNNNNNNNNNNNNN

\bibitem{neyman:scott:1958}
{\sc Neyman, J. and Scott, E.L.} (1958)
A statistical approach to problems of cosmology. {\em
 J. R. Stat. Soc. Ser. B. Stat. Methodol.} 
{\bf 20}, 1--43. 

%\bibitem{norberg:1993}
%{\sc Norberg, R.} (1993)
%Prediction of outstanding liabilities in non-life insurance. {\em ASTIN
%	Bull.} {\bf 23}, 95--115.

%\bibitem{norberg:1999}
%{\sc Norberg, R.} (1999)
%Prediction of outstanding liabilities II. Moel variations and
%	extensions. {\em ASTIN
%	Bull.} {\bf 29}, 5--25.

%PPPPPPPPPPPPPPPPPPPPPPPPPPPPPPPPPPPPPPPPPPPPPPPPPPPPPPPPPPP
%\bibitem{pipiras:taqqu:2000}
%{\sc Pipiras, V. and Taqqu, M.S.}\ (2000)
%The limit of a renewal-reward process with heavy-tailed rewards is not
%a linear fractional stable motion. {\em Bernoulli} {\bf 6}, 607--614.

\bibitem{panjer:1981} 
{\sc Panjer, H.H.} (1981)
Recursive evaluation of a family of compound distributions.
{\em Astin Bull.} {\bf 12}, 22--26.

\bibitem{panjer:wang:1993}
{\sc Panjer, H.H. and Wang, S.} (1993)
On the Stability of Recursive Formulas. 
{\em Astin Bull.} {\bf 23}, 227--258. 

\bibitem{RSST:1999}
{\sc Rolski, T., Schmidli, H., Schmidt, V. and Teugels, J.} (1999)
{\it Stochastic Processes for Insurance and Finance}, Wiley, New York.


\bibitem{rolski:tomanek:2011}
{\sc Rolski, T. and Tomanek, A.}\ (2011)
Asymptotics of conditional moments of the summand in Poisson compounds. 
{\em J. Appl. Probab.} {\bf 48A}, 65--76. 

\bibitem{rolski:tomanek:2013}
{\sc Rolski, T. and Tomanek, A.}\ (2014)
A continuous-time model for claims reserving. 
{\em Applicationes Mathematicae} {\bf 41}, 277--300.

\bibitem{sato:1999}
{\sc Sato, K.-i.} (1999) 
{\em L\'evy Processes and Infinitely Divisible Distributions.} 
Cambridge University Press, Cambridge, UK.

\bibitem{sundt:vernic:2009}
{\sc Sundt, B. and Vernic, R.} (2009)
{\em Recursions for convolutions and compound distributions with
	insurance applications.} 
Springer, Berlin. 

%VVVVVVVVVVVVVVVVVVVVVVVVVVVVVVV
\bibitem{verejones:1970}
{\sc Vere--Jones, D.}\ (1970)
Stochastic models for earthquake occurrences. {\em
	J. R. Stat. Soc. Ser. B. Stat. Methodol.},
{\bf 32}, 1--62.

\bibitem{wheeden:zygmund:1977}
{\sc Wheeden R.L. and Zygmund A.} (1977)
{\em Measure and integral: an introduction to real analysis.} 
CRC Press, Bca Raton.

\bibitem{wilf:1993}
{\sc Wilf H.S.} (1993)
{\em Generatingfunctionology.} 2nd. ed. 
Academic Press, San Diego.

%\bibitem{Wuthrich:Merz:2008}
%{\sc W\"uthrich M.V. and Merz, M.}
%{\em Stochastic Claims Reserving Methods in Insurance.} 
%Wiley, West Sussex. 


\end{thebibliography}
\end{document}